\documentclass[12pt]{article}
\usepackage{tikz}
\usepackage{amsmath}
\usepackage{amssymb}
\usepackage{amsthm}
\usepackage{fullpage}
\usepackage{verbatim}
\usepackage{authblk}
\usepackage{multirow}
\usepackage{mathrsfs}
\usepackage{longtable}

\usepackage[colorlinks]{hyperref}
\usepackage{mathdots}
\usepackage{graphicx,subfigure}
\usepackage[english]{babel}
\usepackage[utf8]{inputenc}
\usepackage{tikz}
\usepackage{mathtools}
\usetikzlibrary{arrows}

\theoremstyle{definition}
\newtheorem{Theorem}[equation]{Theorem}
\newtheorem{Corollary}[equation]{Corollary}
\newtheorem{Lemma}[equation]{Lemma}
\newtheorem{Proposition}[equation]{Proposition}
\newtheorem{Remark}[equation]{Remark}

\newtheorem{Definition}[equation]{Definition}
\newtheorem{Definition-Remark}[equation]{Definition/Remark}
\newtheorem{Example}[equation]{Example}
\newtheorem{Notation}[equation]{Notation}

\numberwithin{equation}{section}
\numberwithin{figure}{section}

\newcommand{\C}{\mathbb{C}}

\newcommand{\Z}{\mathbb{Z}}
\newcommand{\Q}{\mathbb{Q}}

\newcommand{\PP}{\mathbb{P}}

\newcommand{\MC}{\mathit{M}\kern -0.1em\mathrm{1}}
\newcommand{\MCC}{\mathit{M}\kern -0.1em\mathrm{2}}

\newcommand{\mc}[1]{\mathcal{#1}} % short for mathcal
\newcommand{\mbf}[1]{\mathbf{#1}} % short for math bold face
\newcommand{\mt}[1]{\text{#1}}

\begin{document}

\title{Sphericality and Smoothness of Schubert Varieties}

\author[1]{Mahir Bilen Can}
\author[2]{Reuven Hodges}

\affil[1]{{\small Tulane University, New Orleans; mahirbilencan@gmail.com}}    
\affil[2]{{\small University of Illinois Urbana-Champaign; rhodges@illinois.edu}}

\normalsize

\date{\today}
\maketitle

\begin{abstract}
We consider the action of the Levi subgroup of a parabolic subgroup that stabilizes a Schubert variety. 
We show that a smooth Schubert variety is a homogeneous space for a parabolic subgroup, 
or it has a smooth Schubert divisor. Further, we show that all smooth Schubert varieties in a (partial) flag variety of a rank two
simple algebraic group are spherical. 
\vspace{.2cm}

\noindent 
\textbf{Keywords:}  Spherical varieties, Schubert varieties, BP-decomposition, smooth divisors.\\ 
\noindent 
\textbf{MSC: 14M15, 14M27} 
\end{abstract}

\section{Introduction}

Let $Y$ be a complex algebraic variety, and let $G$ be a connected complex reductive algebraic group with an algebraic action $m: G\times Y \to Y$.
If $Y$ is a normal variety, and if a Borel subgroup $B\subseteq G$ has a dense, open orbit in $Y$, then $Y$ is called a {\em spherical $G$-variety}. 
All partial flag varieties, all normal reductive monoids, all homogeneous symmetric spaces and their equivariant compactifications are spherical varieties. 
In this article, we are concerned with the question of which Schubert varieties in $G/P$, where $P$ is a parabolic subgroup, are spherical with respect to an action of a reductive subgroup of $G$.

Given a Schubert variety $Y$ in a flag variety $G/Q$, where $Q$ is a parabolic subgroup, we have a natural choice for a reductive group action. The parabolic subgroups of $G$ that act on $Y$ by left multiplication are well studied~\cite{LMS74}, and can be computed explicitly in terms of root system data. Taking any Levi subgroup of these parabolic subgroups then gives us a natural reductive group action on $Y$.

The main purpose of this article is two-fold. 
First, it is to show that if a smooth Schubert variety is not spherical, then it has a smooth Schubert divisor. 

\begin{Theorem}\label{T:maintheorem}
Let $X_{wB}$ be a smooth Schubert variety in $G/B$ and let $L(w)$ denote the standard Levi subgroup of $P(w)=Stab_G(X_{wB})$. Then $w\in W$ is a maximal element of a parabolic subgroup,
hence $X_{wB}$ is a homogeneous spherical $L(w)$-variety, or $X_{wB}$ has a smooth Schubert divisor. 
\end{Theorem}

This result shows an interesting dichotomy between smoothness and the sphericality properties of Schubert varieties. 
Secondly, we show that every smooth Schubert variety in a (partial) flag variety of a rank two simple algebraic group is spherical. 
Indeed, we identify all the spherical Schubert varieties in types $\mathrm{A}_2$, $\mathrm{A}_3$, $\mathrm{B}_2$, and $\mathrm{G}_2$. 
Along the way we develop a number of methods which allow the determination of the sphericality of a Schubert variety in terms of its divisors. 
The lesson that we learn from working out these small rank examples is that the smooth Schubert varieties of low dimensions are close to being a toric variety,
and therefore are spherical.
In Example~\ref{E:smoothspericalcounterexample}, we show that there are smooth Schubert varieties which are not spherical 
when the rank of the group is bigger than two.

In the case when $G$ is $\mathbf{GL}_n$ and $Q$ is a maximal parabolic subgroup, that is when $G/Q$ is the Grassmann variety, Levi actions have been studied by the second author and Lakshmibai~\cite{HodgesLakshmibai}. This has culminated in a complete characterization of those pairs $(Y,L)$ of Schubert variety $Y$ in the Grassmannian and Levi subgroup $L$ for which $Y$ is a spherical $L$-variety~\cite{HodgesLakshmibaiCL}. 
%In the special case that $L$ is a Levi subgroup of the form $L=\mathbf{GL}_p\times \mathbf{GL}_q$ in 
%$G=\mathbf{GL}_n$, the answer is given by Wyser in~\cite{Wyser13}.

It is shown by Hong and Mok~\cite[Section 3.3]{HongMok} that if $P$ is a maximal parabolic subgroup associated with a long simple root for $G$,
then any smooth Schubert variety in $G/P$ is a homogeneous submanifold of $G/P$ associated to a subdiagram of the marked Dynkin diagram of $G$. 
Since all projective homogeneous spaces (partial flag varieties) are spherical, this result provides us with a large class of smooth Schubert varieties which are spherical.
However, this property does not hold in general. 
Indeed, first of all, smooth Schubert varieties are in general not isomorphic to partial flag varieties.
Moreover, as noted above, not every smooth Schubert variety is spherical with respect to a maximal Levi subgroup action.

One approach to understanding the sphericity of smooth Schubert varieties would be to extend the representation theoretic and combinatorial techniques applied by the second author and Lakshmibai in~\cite{HodgesLakshmibai}, but this would require a case by case approach for each Dynkin type. Even for $G=\mathbf{GL}_n$ and $Q=B$ this becomes a non-trivial problem. 
At this junction, let us mention that there are well-known criteria for detecting smoothness of the Schubert 
varieties, see~\cite{BilleyPostnikov, Billey, Gasharov, LakshmibaiSandhya, OhYoo, RichmondSlofstra, Slofstra15}.
Let us also mention that within the class of spherical varieties are the ``toroidal varieties''
and in our previous work, joint with Lakshmibai, we showed that toroidal Schubert varieties are 
smooth~\cite{Can2019}.

The structure of our paper is as follows.
In Section~\ref{S:Preliminaries}, we set our notation and review some basic facts about spherical varieties. 
The purpose of Section~\ref{S:InversionSets} is to introduce a combinatorial tool for analyzing 
the stabilizer subgroups of Schubert varieties, namely, inversion sets. 
In Section~\ref{S:Review}, we briefly discuss the Billey-Postnikov decompositions, and 
we review some basic facts regarding homogeneous fiber bundles.
In Section~\ref{S:Divisors}, we prove that if $X_{wB}$ is a smooth Schubert variety in $G/B$,
then either $w$ is the maximal element of a parabolic subgroup, or $X_{wB}$ contains 
a smooth Schubert divisor.

In Section~\ref{S:FinalA3} we show that all Schubert varieties in types $\text{A}_2$ and $\text{A}_3$ are spherical.
In Section~\ref{S:FinalB2}, we show that every Schubert variety in type $\text{B}_2$ is spherical.
Finally, in Section~\ref{S:FinalG2}, we show that every Schubert variety in type $\text{G}_2$ is spherical.
%Along the way, we obtain various sphericality criteria for Schubert varieties in terms of their divisors. 

\vspace{.5cm}

\textbf{Acknowledgements.} The second author was partially supported by a grant from Louisiana Board of Regents.

\section{Preliminaries and Notation}\label{S:Preliminaries}

Throughout our paper we work with algebraic varieties and
groups that are defined over $\C$.
By an algebraic variety we mean an 
irreducible separated scheme of finite type. 
Our implicitly assumed symbolism is as follows:

\begin{longtable}{l c l}
$G$ &: & \text{ a connected reductive group;}\\
$T$ &: & \text{ a maximal torus in $G$;}\\
$X^*(T)$ &: & \text{ the group of characters of $T$;}\\
$B$ &: & \text{ a Borel subgroup of $G$ s.t. $T\subset B$;}\\
$U$ &: & \text{ the unipotent radical of $B$ (so, $B=T\ltimes U$);}\\
$\Phi$ &: & \text{ the root system determined by $(G,T)$;}\\
$\Delta$ &: & \text{ the set of simple roots in $\Phi$ relative to $B$;}\\
$W$ &: & \text{ the Weyl group of the pair $(G,T)$ (so, $W=N_G(T)/T$);}\\
$S$ &: & \text{ the Coxeter generators of $W$ determined by $(\Phi,\Delta)$;}\\
$S(v)$ &: & \text{ the support of $v\in W$, that is, $S(v) := \{ s\in S:\ s \leq v\}$;}\\ 
$R$ &: & \text{ the set consisting of $w s w^{-1}$ with $s\in S$, $w\in W$;}\\
$C_w$ &: & \text{ the $B$-orbit through $wB/B$ ($w\in W$) in $G/B$;}\\
$\ell$ &: & \text{ the length function defined by $w\mapsto \dim C_w$ ($w\in W$);}\\
$X_{w}$ &: & \text{ the Zariski closure of $C_w$ ($w\in W$) in $G/B$;}\\
$p_{{\scriptscriptstyle P,Q}}$ &: & \text{ the canonical projection 
$p_{{\scriptscriptstyle P,Q}}:G/P\rightarrow G/Q$
if $P\subset Q \subset G$;}\\
$W_Q$ &: & \text{ the parabolic subgroup of $W$ corresponding
to $Q$, where $B\subset Q$;}\\
$W^Q$ &: & \text{ the minimal length coset representatives
of $W_Q$ in $W$;}\\
$X_{wQ}$ &: & \text{ the image of $X_w$ in $G/Q$ under $p_{{\scriptscriptstyle B,Q}}$.}\\
\end{longtable}

Let $w$ be an element from $W$. 
The {\em right (resp. left) descent set} of $w$ is defined by $D_R(w) = \{s\in S:\ \ell(ws) < \ell(w)\}$ (resp. $D_L(w) = \{s\in S:\ \ell(sw) < \ell(w) \}$).

The {\em right (resp. left) weak order} on $W$ is defined by $u\leq_R w$ (resp. $u\leq_L w$) if there exist 
$s_1,\dots, s_k\in S$ such that $w= u s_1\cdots s_k$  (resp. $w= s_1\cdots s_k u$) and 
$\ell(u s_1\cdots s_i) = \ell(u) + i$ for $1\leq i \leq k$ (resp. $\ell(s_1\cdots s_i u) = \ell(u) + i$ for $1\leq i \leq k$). We will write $x \lessdot_R y$ to indicate that $x$ is \emph{covered} by $y$, that is $x <_R y$ and there is no $z$ such that $x <_R z <_R y$.

\vspace{.25cm}
The Weyl group $W$ operates on $X^*(T)$ and the action leaves $\Phi$ stable.
In particular, $\Phi$ spans a not necessarily proper euclidean subspace $V$ of $X^*(T)\otimes_\Q \Z$;
$\Delta$ is a basis for $V$. The elements of $S$ are 
called simple reflections and the elements of $R$ are 
called reflections. The elements of $S$ are in one-to-one correspondence
with the elements of $\Delta$, and furthermore, $S$ generates $W$.

\vspace{.25cm}

When there is no danger for confusion,
for the easing of our notation, we will use the notation $w$ 
to denote a coset in $N_G(T)/T$ as well as to denote 
its representative $n_w$ in the normalizer subgroup $N_G(T)$ in $G$.
The {\em Bruhat-Chevalley decomposition} of $G$ is
$$G= \bigsqcup_{w\in W} B w B$$ and the 
Bruhat-Chevalley order on $W$ is defined by 
$w \leq v \iff B w B \subseteq \overline{B v B}$.
Here, the bar indicates Zariski closure in $G$.
Then $(W,\leq)$ is a graded poset with grading given by $\ell$. 
The order $\leq$ induces an order on $W^Q$, denoted by $\leq_Q$:
\[
wW^Q \leq_Q v W^Q\ \text{ in $W^Q$} \iff w' \leq v' \text{ in $W$},
\] 
where $w'$ (resp. $v'$) is the minimal length left coset representative of $w$ (resp. of $v$).
We will call $\leq_Q$ the {\em Bruhat-Chevalley order on $W^Q$}, and if confusion is unlikely, 
then we will denote $\leq_Q$ simply by $\leq$.

\vspace{.25cm}

Let $Q$ be a parabolic subgroup of $G$ with $B\subset Q$. 
Then $Q$ is called a \emph{standard parabolic subgroup} 
with respect to $B$.
The canonical projection $p_{B,Q}$ restricts to an isomorphism
$p_{B,Q} \vert_{C_w} : C_w \rightarrow BwQ/Q$,
hence it restricts to a birational morphism 
$p_{B,Q} \vert_{X_w} : X_w \rightarrow X_{wQ}$ for $w \in W^Q$.
In this case, that is $w\in W^Q$, the preimage in $G/B$ of $X_{wQ}$ is equal 
to $X_{ w w_{0,Q}}$, where 
$w_{0,Q}$ denotes the unique maximal element of $(W_Q,\leq)$, 
and the restriction $p_{B,Q} \vert_{X_{ w w_{0,Q}}} : X_{ w w_{0,Q}} \rightarrow X_{wQ}$
is a locally trivial fibration with generic fiber $Q/B$.

The standard parabolic subgroups with respect to $B$ 
are determined by the subsets of $S$; 
let $I$ be a subset in $S$ and define $P_I$ by 
\begin{align}\label{A:standard parabolic}
P_I := B W_I B,
\end{align}
where $W_I$ is the subgroup of $W$ generated by the elements in $I$. 
Then $P_I$ is a standard parabolic subgroup with respect to $B$. 
Any parabolic subgroup in $G$ is 
conjugate-isomorphic to exactly one such $P_I$. 
In the next paragraphs we will briefly 
review the structure of $P_I$.

Let $I$ be as in (\ref{A:standard parabolic}). 
By abusing the notation, we denote the corresponding subset in $\Delta$ by $I$ as well. 
Denote by $\Phi_I$ the subroot system in $\Phi$ that is generated by $I$. 
The intersection $\Phi^+ \cap \Phi_I$, which we denote by $\Phi_I^+$, 
forms a system of positive roots for $\Phi_I$. 
In a similar way we denote $\Phi^-\cap \Phi_I$ by $\Phi_I^-$. 
Then $\Phi_I = \Phi_I^+ \cup \Phi_I^-$. In this notation, we have 
\begin{align*}
W^I = \{ x\in W :\ \ell(xw) = \ell(x) + \ell(w) \ \text{ for all } w\in W_I \}= \{ x\in W :\  x(\Phi_I^+)\subset \Phi^+\}.
\end{align*}
In other words, $W^I$ is the set of 
minimal length right coset representatives of $W_I$ in $W$.

For a simple root $\alpha$ from $\Delta$,
we denote by $U_\alpha$ the corresponding one-dimensional unipotent subgroup.
Let $L_I$ and $U_I$ be the subgroups defined by 
$L_I := \langle T, U_\alpha :\ \alpha \in \Phi_I \rangle$ and 
$U_I := \langle U_\alpha:\ \alpha \in \Phi^+ \setminus \Phi_I^+ \rangle$.
Then $L_I$ is a reductive group and $U_I$ is a unipotent group. 
The Weyl group of $L_I$ is equal to $W_I$. 
The relationship between $L_I,U_I$, and $P_I$ is given by 
$P_I = L_I \ltimes U_I$, where $U_I$ is the unipotent radical of $P_I$. 
We will refer to $L_I$ as the standard Levi factor of $P_I$. 
In the most extreme case that $I=\emptyset$,
so $P_I=B$, the Levi factor is the maximal torus $T$. 
In general, Levi subgroups are defined as follows. 
Let $P$ be a parabolic subgroup and let 
$\mc{R}_u(P)$ denote its unipotent radical.
A subgroup $L\subset P$ is called a Levi subgroup if 
$P= L \ltimes \mc{R}_u(P)$ holds true.

We apply the following notational convention throughout our paper:
given a parabolic subgroup $P$
with Levi factor $L$, the Weyl group of $L$ is denoted by $W_P$. 
In the same way, the set of minimal length coset representatives 
of $W_P$ in $W$ will be denoted $W^P$.

Clearly, the minimal parabolic subgroups containing $B$ are in one-to-one correspondence with 
the set of simple roots. 
The following observation, which is due to Kempf~\cite{Kempf1976}, will be useful in the sequel.
\begin{Lemma}\label{L:Kempf}
Let $Q$ be a minimal parabolic subgroup corresponding to a simple root $\alpha$,
let $s_\alpha$ denote the simple Coxeter generator corresponding to $\alpha$. 
Given a Schubert variety $X=X_{wB}$ the following assertions hold: 
\begin{enumerate}
\item $X$ is a $\PP^1$-bundle over its image $p_{B,Q}(X)$ if and only if 
$$
\ell( w \cdot s_\alpha) = \ell(w) - 1.
$$ 
\item $X$ is mapped birationally onto its image $p_{B,Q}(X)$ if and only if 
$$
\ell( w \cdot s_\alpha) = \ell(w) + 1.
$$ 
\item In any case, the fiber of the restriction of $p_{B,Q}$ onto $X$ is either a point or a  
projective line. 
\end{enumerate}
\end{Lemma}

\subsection{Background on spherical varieties.}

Let $V$ be a $G$-module. 
The vector space $V$ 
is called {\em multiplicity-free} if for any dominant weight $\lambda \in X^*(T)$
the following inequality holds true:
$$
\dim \mt{Hom}_G (V(\lambda), V) \leq 1,
$$ 
where $V(\lambda)$ is the irreducible representation of $G$ with highest weight
$\lambda$. 
A proof of the following result on the characterization of spherical varieties can be found in~\cite{Perrin}.
\begin{Theorem}\label{T:Perrin}
Let $X$ be a normal $G$-variety, and let $B$ denote a Borel subgroup of $G$. 
The following conditions are equivalent:
\begin{enumerate}
\item all $B$-invariant rational functions on $X$ are constant functions;
\item the minimal codimension of a $B$-orbit in $X$ is zero; 
\item $X$ is a union of finitely many $B$-orbits; 
\item $B$ has a dense open orbit in $X$.
\item If $X$ is quasi-projective, then these conditions are equivalent to the following condition: 
For every $G$-linearized line bundle $\mc{L}$ on $X$, the $G$-module
$H^0(X,\mc{L})$ is multiplicity-free.
\end{enumerate}
\end{Theorem}

Let $X$ be a spherical $G$-variety, and let $x_0$ be a point in general position in $X$. 
Then $G\cdot x_0$ is a spherical homogeneous $G$-variety of the form $G/H$ for some closed subgroup $H$. 
We will call such subgroups {\em spherical subgroups in $G$}. 
Let us mention that $L=\mathbf{GL}_p\times \mathbf{GL}_q$
is one of the few examples of a Levi subgroup in a semisimple group which
is a spherical subgroup in its ambient reductive group $\mathbf{GL}_{p+q}$.
The classification of reductive spherical subgroups in $G$ is due to Brion~\cite{Brion87},
Kr\"amer~\cite{Kramer79}, and Mikityuk~\cite{Mikityuk86}. 

\begin{comment}
In his paper, Brion classified the spherical Levi subgroups in characteristic 0. 
Later, in~\cite{Brundan98}, Brundan extended Brion's classification over fields with positive characteristic. 
Let $L$ be a Levi subgroup and let $G'$ denote the derived subgroup
of $G$. Let $\prod_{i=1}^r G_i$ be the decomposition of
$G'$ as a commuting product of simple factors and 
set $L_i:= L \cap G_i$. Then $L$ is spherical in $G$ if and
only if, for each $i$, either $G_i = L_i$ or $(G_i,L_i')$ 
is one of 
\begin{align}\label{A:Brundan's list}
(\mathrm{A}_n,\mathrm{A}_m \mathrm{A}_{n-m-1}), 
(\mathrm{B}_n,\mathrm{B}_{n-1}), 
(\mathrm{B}_n,\mathrm{A}_{n-1}), 
(\mathrm{C}_n,\mathrm{C}_{n-1}), \\ \notag
(\mathrm{C}_n,\mathrm{A}_{n-1}), 
(\mathrm{D}_n,\mathrm{D}_{n-1}), 
(\mathrm{D}_n,\mathrm{A}_{n-1}), 
(\mathrm{E}_6,\mathrm{D}_{5}), 
(\mathrm{E}_7,\mathrm{E}_{6}).
\end{align}
Note that the property that a subgroup $H$ in $G$ is spherical is 
preserved by the isogenies of $G$, so 
any representative of the corresponding 
root datum in (\ref{A:Brundan's list}) 
gives a pair of reductive group and a 
spherical Levi subgroup.
\end{comment}

\subsection{Diagonal actions.}\label{SS:diagonal}

Let $P$ and $Q$ be two parabolic subgroups in a semisimple algebraic group $G$. 
If $L$ is a Levi subgroup of $P$, by~\cite[Lemma 5.3]{AvdeevPetukhov}, $G/Q$ is a spherical $L$-variety 
if and only if the diagonal action of $G$ on $G/P\times G/Q$ is spherical. 
As Littelmann showed in~\cite{Littelmann}, the study of such diagonal actions is very useful in representation theory. 
For $G\in \{\mbf{SL}_n,\mbf{Sp}_n\}$, Magyar, Zelevinsky, Weyman~\cite{MWZ1, MWZ2} classified all
parabolic subgroups $P_1,\dots, P_k\subset G$ such that the variety $\prod_{i=1}^r G/P_i$ has only finitely 
many diagonal $G$-orbits. For $P_1=B$, this classification amounts to the classification of 
spherical varieties of the form $\prod_{i=2}^r G/P_i$ with respect to the diagonal $G$-actions. 
In~\cite{Stembridge}, Stembridge presents the complete list of the diagonal spherical actions of $G$ on the products of the form $G/P\times G/Q$.
It follows from Stembridge's list that, if both of the parabolic subgroups $P$ and $Q$ are properly contained in $G$, 
then a Levi subgroup $L$ of $P$ can act spherically on $G/Q$ only if the simple factors of the semisimple part of $G$ avoid the types 
$\text{E}_8,\text{F}_4$, and $\text{G}_2$.

\begin{Remark}\label{R:Littelmann's}
We know from the classification schematics that, in type $\text{A}_n$, every Levi factor of every maximal parabolic subgroup $P$ of $G$ acts spherically on every Grassmannian of $G$.
In fact, according to~\cite{MWZ1}, for $G=\mathbf{SL}_{n+1}$, and all $I,J\subset S=\{s_1,\dots, s_n\}$,  
the variety $G/P_I\times G/P_J$ is a spherical $G$-variety if and only if  up to interchanging of $I$ and $J$, one has exactly 
one of the following possibilities: 
\begin{enumerate}
\item $I^c = \emptyset, \{s_1\}$ or $\{s_n\}$, 
\item $I^c = \{s_2\}$ or $\{s_{n-1}\}$ and $|J^c | = 2$,
\item $| I^c |= |J^c |=1$,
\item $|I^c |= 1$ and $J^c = \{ s_1,s_j \}, \{ s_j,s_{j+1} \}$ or $\{s_j,s_n\}$ ($1<j<n$).
\end{enumerate} 
In type $\text{D}_n$, every Levi factor of the maximal parabolic $P$ associated with 
$\omega_1$ or $\omega_n$ acts spherically on $G/P$, where $\omega_1$ and $\omega_n$ are the first and last fundamental weights. 
More precisely, according to~\cite{Stembridge}, for $G=\mathbf{Spin}_{2n}$ ($n\geq 4$) and all proper $I,J\subset S=\{s_1,\dots, s_n\}$,  
the variety $G/P_I\times G/P_J$ is a spherical $G$-variety if and only if  up to interchanging of $I$ and $J$, one has exactly 
one of the following possibilities: 
\begin{enumerate}
\item $I^c = \{s_n\}$ and $J^c = \{s_i\}, \{ s_1,s_i \}$ or $\{ s_2,s_i \}$ ($1\leq i\leq n$).
\item $I^c = \{s_1\}$ or $I^c = \{s_2\}$ and $J^c \subsetneq \{s_1,s_2,s_n\}$, $J^c \subseteq \{s_{n-1},s_n\}$,
or $J^c= \{s_{n-2}\}$, or 
\item ($n=4$ only) $I^c = \{s_1\}$ and $J^c= \{s_2,s_3\}$ or $I^c = \{s_2\}$ and $J^c= \{s_1,s_3\}$.
\end{enumerate} 
Here, the diagram labeling is given by 
\begin{figure}[htp]
\begin{center}

\begin{tikzpicture}[scale=.6pt]

\node at (-6,0) (a1) {$1$};
\node at (-4,0) (a3) {$3$};
\node at (-4,2) (a2) {$2$};
\node at (-2,0) (a4) {$4$};
\node at (0,0) (a5) {};
\node at (.5,0) (a6) {$\cdots$};
\node at (1,0) (a7) {};
\node at (3,0) (a8) {$n$};
%\node at (6,0) (a2) {$2$};

\draw[-, thick] (a1) to (a3);
\draw[-, thick] (a2) to (a3);
\draw[-, thick] (a3) to (a4);
\draw[-, thick] (a4) to (a5);
\draw[-, thick] (a7) to (a8);
%\draw[-, thick] (a3) to (a2);
%\draw[-, thick] (a4) to (a1);

\end{tikzpicture}
\caption{The labeling of the nodes of $\text{D}_n$.}\label{F:Dn}
\end{center}
\end{figure}
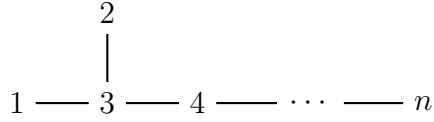
\end{Remark}

\begin{comment}

\begin{Remark}
Let $w_1$ and $w_2$ be two elements from $W$.
Then $w_1\leq w_2 \iff w_1^{-1} \leq w_2^{-1}$. 
In particular, every lower interval $[id, w]$ in $(W,\leq)$ is isomorphic to the lower interval $[id, w^{-1}]$. 
Let $X_{wB}$ be a smooth Schubert variety in $G/B$.
It is well-known that $X_{wB}$ is a rationally smooth variety if and only if the Kazhdan-Lusztig polynomial $P_{e,w}(q)$ 
is 1, that is, $P_{e,w}(q)=1$. (This is shown by Kazhdan and Lusztig.)
It is shown in~\cite{Delanoy2006} that, if $[id, w]$ and $[id,w']$ are two lower intervals in the Coxeter groups $W$ and $W'$, respectively,
such that there exists a poset isomorphism $\phi : [id,w]\to [id,w']$, then the corresponding Kazhdan-Lusztig polynomials are the same,
\begin{align}\label{A:thesame}
P_{x,y}(q)= P_{\phi(x),\phi(y)}(q)\ \text{ for every } x,y\in [id,w].
\end{align}
In particular, it follows from (\ref{A:thesame}) that $X_{wB}$ is a rationally smooth Schubert variety in $G/B$
if and only if $X_{w'B'}$ is a rationally smooth Schubert variety in $G'/B'$, where $G'$ is a connected semisimple algebraic group with Weyl group $W'$. 
Finally, let us point out another well-known result, due to Peterson, which states that 
if $G$ is a simply-laced semisimple algebraic group, then every rationally smooth Schubert variety is smooth. 
Of course, a smooth Schubert variety is always rationally smooth. 
\end{Remark}

\end{comment}

\section{Inversion Sets}\label{S:InversionSets}

We begin with a basic definition.

\begin{Definition}
For $v\in W$, the {\em right inversion set of $v$} is defined by $I(v) := \Phi^+ \cap v^{-1} ( \Phi^-)$. 
The {\em (left) inversion set of $v$} is defined by $N(v):= I(v^{-1})= \Phi^+ \cap v(\Phi^-)$.
\end{Definition}

We note that $I(v)$ and $N(v)$ need not be disjoint. 
In fact, it is easy to show the following simple fact (by using Proposition~\ref{P:B_0}):
\begin{Lemma}\label{L:involutions-inversions}
For $w\in W$, 
the left and the right inversion sets of $w$ are the same if and only if $w$ is an involution, that is, $w^2 = id$.
\end{Lemma}

\begin{Definition}
Let $V$ denote the Euclidean space that is spanned by a root system $\Phi$, and let $A$ be a subset of $\Phi^+$.
\begin{enumerate}
\item $A$ called {\em convex} if $\text{cone}(A)\cap \Phi=A$, where $\text{cone}(A)$ is the cone generated by $A$ in $V$; 
\item $A$ is called {\em coconvex} if $\Phi^+ \setminus A$ is convex; 
\item $A$ is called {\em closed} if for every $\alpha$ and $\beta$ from $A$, we have $\text{cone}(\{\alpha,\beta\})\cap \Phi \subseteq A$;
\item $A$ is called {\em coclosed} if $\Phi^+\setminus A$ is closed. 
\end{enumerate}
Finally, $A$ will be called {\em biclosed} if it is closed and coclosed at the same time;
$A$ will be called {\em biconvex} if it is convex and coconvex at the same time. 
\end{Definition}

\begin{comment}

The following two propositions are originally due to Dyer~\cite{Dyer2011}. 
\begin{Proposition}
Let $A$ be a subset of $\Phi^+$. 
The following assertions are equivalent. 
\begin{enumerate}
\item $A$ is equal to $N(w)$ for some $w\in W$; 
\item $A$ is biclosed;
\item $A$ is biconvex.
\end{enumerate}
\end{Proposition}

\begin{Example}\label{E:tosetupnotation}
Let $\alpha_i$ ($i\in \{1,2,3\}$) denote the standard simple root $\varepsilon_i -\varepsilon_{i+1}$
in $\Phi$, the root system for $(\mathbf{SL}_4,\mathbf{T})$, where $\mathbf{T}$ is the diagonal torus 
in $\mathbf{SL}_4$. The Weyl group of $\Phi$ is the symmetric group of permutations, $S_4$. 
Let $w$ denote the element $4231$. Then the inversion set of $w$ is given by 
\[
N(w)= \{ \alpha_1,\alpha_1+\alpha_2,\alpha_1+\alpha_2+\alpha_3, \alpha_2+\alpha_3,
\alpha_3\}.
\]
Then $\Phi^+\setminus N(w) = \{\alpha_2\}$. 
It is now obvious that $N(w)$ is biclosed.
\end{Example}

\end{comment}

\begin{Proposition}\label{P:B_0}
Let $\mathcal{B}_0:=\mathcal{B}_0(\Phi^+)$ denote the poset of biclosed subsets of $\Phi^+$ ordered with respect to inclusion of subsets.
Then the map $N: (W,\leq_R) \longrightarrow \mathcal{B}_0$ defined by $w\mapsto N(w)$ ($w\in W$) is a poset isomorphism between the right
weak order on $W$ and $(\mathcal{B}_0,\subseteq)$.
\end{Proposition}

Let $K$ be a subset of $S$, and let $v$ be an element from $W^K$. 
The stabilizer in $G$ of the Schubert subvariety $X_{vP_K}$ is the parabolic subgroup generated by $B$ and $s_\alpha$, where $\alpha \in \Delta \cap v( \Phi^- \cup \Phi_K)$.

\begin{Notation}
Let $w$ be an element from $W$. We will denote the subset of simple roots of $N(w)$ by $N_\Delta(w)$.
In other words, $N_\Delta(w) := \Delta \cap N(w)$. 
%Likewise, we define $I_\Delta (w):= \Delta \cap I(w)$.
We will denote $Stab_G(X_{wB})$ by $P(w)$, and likewise, 
we will denote the Levi subgroup of $P(w)$ containing $T$ by $L(w)$.
\end{Notation}

The proof of the following lemma is evident from definitions, so we omit it. 
\begin{Lemma}\label{L:evident}
For $w\in W$, $P(w)$ (resp. $L(w)$) is generated by $B$ (resp. $T$) 
and $s_\alpha$, where $\alpha \in N_\Delta(w)$.
\end{Lemma}

Let $w$ be an element from $W$ such that $w= vu$ for some $v,u\in W$ with $\ell(w) = \ell(v)+\ell(u)$. 
Hence, $v\leq_R w$. By~\cite[Proposition 2.1]{HohlwegLabbe}, we know that  
\begin{align}\label{A:wg}
N(w) = N(v) \sqcup v (N(u)).
\end{align}
In particular, we know that $N_\Delta(v) \subseteq N_\Delta(w)$.
\begin{Corollary}\label{C:P(v)inP(w)}
Let $v$ and $w$ be two elements from $W$. 
If $v\leq_R w$, then $P(v) \subseteq P(w)$. 
Furthermore, if $w=vu$ for some $u\in W$, then the inclusion $P(v)\subseteq P(w)$ 
is strict if and only if there exists a root $\alpha \in N(u)$ such that $v(\alpha) \in \Delta$.
\end{Corollary}
\begin{proof}
This is a straightforward consequence of Lemma~\ref{L:evident} and (\ref{A:wg}). 
\end{proof}

%\begin{Example}
%If $v= 1324$ and $w=3412$ are from $S_4$, then $v <_R w$, however, $P(v)=P(w)$.
%\end{Example}

%\begin{Definition}
%Let $X$ be a $G$-variety. $X$ is called {\em horospherical} if the stabilizer of a point in a 
%general position contains a maximal unipotent subgroup of $G$. 
%\end{Definition}

\begin{Lemma}\label{L:U_w^-}
Let $X_{wB}$ be a Schubert variety in $G/B$. 
If the unipotent group $U_w^- := \langle U_\alpha :\ \alpha \in I(w) \rangle$ is contained in $L(w)$, 
then $X_{wB}$ is a spherical $L(w)$-variety. 
\end{Lemma}

\begin{proof}
It is well-known that the open $B$-orbit $C_{wB}$ in $X_{wB}$ is isomorphic, as a variety, 
to the unipotent subgroup $U_w^-$ that is directly spanned by the root subgroups $U_\alpha$, $\alpha \in \Phi^+ \cap w^{-1}(\Phi^-)$.
In other words, $U_w^- = \langle U_\alpha :\ \alpha \in I(w) \rangle$ is isomorphic to the open Bruhat cell of $X_{wB}$. 
This isomorphism is given by $u \mapsto uwB/B$ ($u\in U_w^-$), see~\cite[Theorem 14.12]{Borel} where $U_w^-$ is denoted by $U_w'$. 

Lemma~\ref{L:evident} implies that $L(w)$ is generated by $T$ and the root subgroups $U_{\pm \beta}$, where $\beta \in \Delta \cap N(w)$ in $Stab_G(X_{wB})$.
In particular, the subgroup $\langle U_\beta :\ \beta \in \Delta \cap N(w) \rangle$ is a maximal unipotent subgroup in $L(w)$. 
If $U_w^-$ is contained in $L(w)$, then it is contained in the maximal unipotent subgroup 
$\langle U_\beta :\ \beta \in \Delta \cap N(w) \rangle$. 
But then $L(w)$ has an open dense orbit in $X_{wB}$. 
This finishes the proof of our assertion. 
\end{proof}

\begin{Remark}
Lemma~\ref{L:U_w^-} can be paraphrased as follows:
if $I(w) \subseteq \Phi_{N_\Delta(w)}$, then $X_{wB}$ is $L(w)$-spherical. 
Clearly, in this case, we must have $I(w)\cap \Delta \subseteq N_\Delta (w)$.
\end{Remark}

\begin{Example}\label{E:S4}
Let $\alpha_i$ ($i\in \{1,2,3\}$) denote the standard simple root $\varepsilon_i -\varepsilon_{i+1}$
in $\Phi$, the root system for $(\mathbf{SL}_4,\mathbf{T})$, where $\mathbf{T}$ is the diagonal torus 
in $\mathbf{SL}_4$. The Weyl group of $\Phi$ is the symmetric group of permutations, $S_4$. 
The elements $w\in S_4$ that satisfy the hypothesis of Lemma~\ref{L:U_w^-} are given by
\[
1234,2134,1324,1243, 1432, 2143, 3214, 4321.
\]
In Figure~\ref{F:S4}, we depict $(S_4,\leq_{R})$ together with 
the simple roots of the inversion sets of each of its elements.

\begin{figure}[htp]
\begin{center}

\begin{tikzpicture}[scale=.2pt]

\node at (0,0) (a) {${1234}\atop{\emptyset}$};

\node at (-8,7) (b1) {${1243}\atop{\alpha_3}$};
\node at (0,7) (b2) {${1324}\atop{\alpha_2}$};
\node at (8,7) (b3) {${2134}\atop{\alpha_1}$};

\node at (-16,14) (c1) {${1423}\atop{\alpha_3}$};
\node at (-8,14) (c2) {${1342}\atop{\alpha_2}$};
\node at (0,14) (c3) {${2143}\atop{\alpha_1,\alpha_3}$};
\node at (8,14) (c4) {${3124}\atop{\alpha_2}$};
\node at (16,14) (c5) {${2314}\atop{\alpha_1}$};

\node at (-20,21) (d1) {${1432}\atop{\alpha_2,\alpha_3}$};
\node at (-12,21) (d2) {${4123}\atop{\alpha_3}$};
\node at (-4,21) (d3) {${2413}\atop{\alpha_1,\alpha_3}$};
\node at (4,21) (d4) {${3142}\atop{\alpha_2}$};
\node at (12,21) (d5) {${3214}\atop{\alpha_1,\alpha_2}$};
\node at (20,21) (d6) {${2341}\atop{\alpha_1}$};

\node at (-16,28) (e1) {${4132}\atop{\alpha_2,\alpha_3}$};
\node at (-8,28) (e2) {${4213}\atop{\alpha_1,\alpha_3}$};
\node at (0,28) (e3) {${3412}\atop{\alpha_2}$};
\node at (8,28) (e4) {${2431}\atop{\alpha_1,\alpha_3}$};
\node at (16,28) (e5) {${3241}\atop{\alpha_1,\alpha_2}$};

\node at (-8,35) (f1) {${4312}\atop{\alpha_2,\alpha_3}$};
\node at (0,35) (f2) {${4231}\atop{\alpha_1,\alpha_3}$};
\node at (8,35) (f3) {${3421}\atop{\alpha_1,\alpha_2}$};

\node at (0,42) (g) {${4321}\atop {\alpha_1,\alpha_2,\alpha_3}$};

\draw[-, thick] (a) to (b1);
\draw[-, thick] (a) to (b2);
\draw[-, thick] (a) to (b3);

\draw[-, thick] (b1) to (c1);
\draw[-, thick] (b1) to (c3);

\draw[-, thick] (b2) to (c2);
\draw[-, thick] (b2) to (c4);

\draw[-, thick] (b3) to (c3);
\draw[-, thick] (b3) to (c5);

\draw[-, thick] (c1) to (d1);
\draw[-, thick] (c1) to (d2);

\draw[-, thick] (c2) to (d1);
\draw[-, thick] (c2) to (d4);

\draw[-, thick] (c3) to (d3);

\draw[-, thick] (c4) to (d4);
\draw[-, thick] (c4) to (d5);

\draw[-, thick] (c5) to (d5);
\draw[-, thick] (c5) to (d6);

\draw[-, thick] (d1) to (e1);

\draw[-, thick] (d2) to (e1);
\draw[-, thick] (d2) to (e2);

\draw[-, thick] (d3) to (e2);
\draw[-, thick] (d3) to (e4);

\draw[-, thick] (d4) to (e3);

\draw[-, thick] (d5) to (e5);

\draw[-, thick] (d6) to (e4);
\draw[-, thick] (d6) to (e5);

\draw[-, thick] (e1) to (f1);

\draw[-, thick] (e2) to (f2);

\draw[-, thick] (e3) to (f1);
\draw[-, thick] (e3) to (f3);

\draw[-, thick] (e4) to (f2);

\draw[-, thick] (e5) to (f3);

\draw[-, thick] (f1) to (g);
\draw[-, thick] (f2) to (g);
\draw[-, thick] (f3) to (g);

\end{tikzpicture}
\caption{The right weak order on $S_4$.}\label{F:S4}
\end{center}
\end{figure}
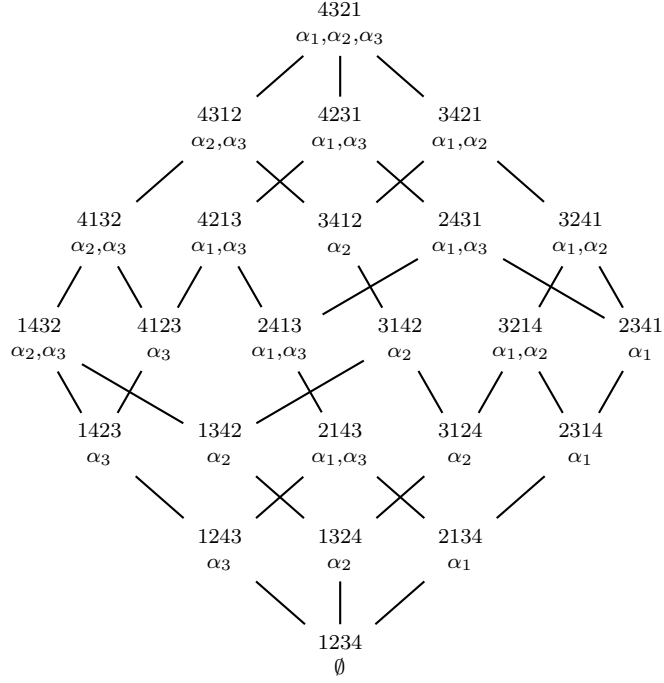

\begin{Proposition}\label{P:U_w^-}
Let $X_{wB}$ be a Schubert variety in $G/B$. 
If one of the following equivalent conditions are satisfied,  
then $X_{wB}$ is a homogeneous spherical $L(w)$-variety. 
\begin{enumerate}
\item[(1)] The set of elements in the lower interval $[id,w^{-1}]_L$ in the left weak order on $W$ 
is equal to the set of elements in the lower interval $[id,w^{-1}]$ in the Bruhat-Chevalley order on $W$.
\item[(2)] The set of elements in the lower interval $[id,w]_R$ in the right weak order on $W$ 
is equal to the set of elements in the lower interval $[id,w]$ in the Bruhat-Chevalley order on $W$.
\end{enumerate}
\end{Proposition}

\begin{proof}
The equivalence of (1) and (2) follows from the well-known isomorphism $w\mapsto w^{-1}$ between 
the posets $(W,\leq_L) \cong (W,\leq_R)$.
We proceed with the assumption that (2) holds true. 
Since $[id,w]$ contains all simple reflections from the support of $w$, we see that $S(w)\subseteq [id, w]_R$. 
But a simple reflection $s$ is contained in the interval $[id,w]_R$ if and only if $s$ is a left descent of $w$. 
In other words, in our case, $D_L(w) = S(w)$. This holds true if and only if $w$ is the maximal element of the parabolic subgroup $W_J$,
where $J=S(w)$. Then we get $X_{wB} = P(w) B/B = L(w) B/B \cong L(w)/L(w)\cap B$.
It follows that $X_{wB}$ is a (projective) homogeneous $L(w)$-variety, hence it is a spherical $L(w)$-variety. 
\end{proof}

The entire inversion set $N(w)$ can be computed using~\cite[Proposition 2.1(i)]{HohlwegLabbe}. 
Here is a simple result that shows how to compute $N_\Delta(w)$. 
\begin{Proposition}\label{P:simpleresult}
Let $w$ be an element from $W$. 
Then we have the following equalities: 
$$
\displaystyle N_\Delta (w) = \{ \alpha \in \Delta :\ s_\alpha \leq_R w \} = \bigcup_{v\lessdot_R w} N_\Delta(v).
$$
\end{Proposition}
\begin{proof}
Let $\alpha$ be an element from $N_\Delta (w)$. 
By Proposition~\ref{P:B_0}, we know that $N(s_\alpha)$ is a subset of $N(w)$. 
But $N(s_\alpha) = N_\Delta(s_\alpha)$, hence, we get the first equality. 
Our second equality is an immediate consequence of the first one. 
\end{proof}

\end{Example}

\section{BP Decomposition and Homogeneous Fiber Bundles}\label{S:Review}

Let $Q$ be a parabolic subgroup containing a parabolic subgroup $P$, and let $w$ be an element from $W^P$. 
Then $w$ can be expressed, in a unique way, as a product of the form $w= vu$, where $v\in W^Q$ and $u\in W^P \cap W_Q$. 
This decomposition of $w$ is called the {\em parabolic decomposition of $w$ with respect to $Q$}. 
If any of the following equivalent conditions is satisfied, then the parabolic decomposition $w=vu$ is called a {\em right BP decomposition of $w$}:
\begin{enumerate}
\item $u$ is the maximal length element in $[id,w]\cap W_Q$.
\item The Poincar\'e polynomial of the Schubert variety $X_{wP}$ is the product of the Poincar\'e polynomials of the 
Schubert varieties $X_{vQ}$ and $X_{uP}$. 
\item The set $S(v) \cap W_Q$ is contained in $D_L(u)=\{ s\in S:\ \ell(su) <  \ell(u) \}$.
\item The projection $p_{P,Q} \vert_{X_{wP}} : X_{wP} \to X_{vQ}$ is a fiber bundle with fibers isomorphic to $X_{uP}$.
\end{enumerate}

The ``left BP decomposition of $w$'' is defined similarly, by using the minimal length right coset representatives instead of the minimal length left coset representatives. 
From now on we will refer to a right BP decomposition simply as a {\em BP decomposition}.
A BP decomposition $w= vu$ ($w\in W^P$, $v\in W^Q$) is said to be a {\em Grassmann BP decomposition of $w$} if $|S(w)| = |S(v)| +1$. If $S(w)$ is equal to $S$, then $Q$ is a maximal parabolic subgroup.

\vspace{.25cm}

Next, we will briefly review homogeneous fiber bundles from a more general viewpoint.   
Let $K$ be a connected algebraic subgroup of a connected algebraic group $L$, and let $Z$ be a normal $K$-variety. 
A {\em homogeneous bundle over $L/K$ associated with $Z$}, denoted by $L*_K Z $,
is the quotient variety $(L\times Z) / K$, where the action of $K$ is given by 
\[
k\cdot (l,z) = ( lk^{-1}, kz)\ \text { for every } k\in K, l\in L,z\in Z.
\]
Then $L$ also acts on $L*_KZ$ via its left multiplication on the first factor of $L\times Z$. 
The natural projection $L*_K Z\to L/K$ is $L$-equivariant.

A discussion of the following results as well as the references to their proofs can be found in~\cite[Appendix 1]{Avdeev15}.
\begin{Proposition}\label{P:A1}
 If $X$ is an $L$-variety and if it admits an $L$-equivariant morphism $\Psi: X \to L/K$, then $X \cong L*_K Z$, where $Z= \Psi^{-1}(o)$. Here, $o$ stands for the base point $idK$ in $L/K$.
\end{Proposition}
\begin{Proposition}\label{P:A2}
 The variety $X \cong L*_K Z$ is smooth (resp. normal) if and only if $Z$ is smooth (resp. normal).
\end{Proposition}
\begin{Proposition}\label{P:A3}
Let $P$ be a parabolic subgroup of $G$ and let $Z$ be a $P$-variety.
The variety $G*_PZ$ is complete if and only if $Z$ is complete. 
\end{Proposition}

Let $Z$ be a $P$-variety, where $P$ is a parabolic subgroup.
Following Avdeev, we will say that $Z$ is a {\em spherical $P$-variety} if 
$Z$ is a spherical $L$-variety, where $L$ is a Levi subgroup of $P$.   
\begin{Proposition}\cite[Proposition 6.3]{Avdeev15}\label{P:Avdeev}
Let $Z$ be a $P$-variety and consider the $G$-variety $X=G*_PZ$.
Then we have 
\begin{enumerate}
\item[(a)] $Z$ is a spherical $P$-variety if and only if $X$ is a spherical $G$-variety.
\item[(b)] $Z$ is a wonderful $P$-variety if and only if $X$ is a wonderful $G$-variety. 
\end{enumerate}
\end{Proposition}

Here, by a ``wonderful variety'' one refers to a certain smooth complete spherical variety with some special geometric properties. 
Since we are not going to need this notion for the results of our paper, we will not introduce it; we refer the reader to Avdeev's paper.
These basic results about homogeneous bundles lead us to the following criterion.

\begin{Proposition}\label{R:followsfromavdeev}
Let $X_{wB}$ be a smooth Schubert variety with a Grassmannian BP decomposition $w= vu$ ($v\in W^J,u\in W_J$).
Let $L$ denote the standard Levi subgroup of the stabilizer of $X_{vP_J}$, 
and let $K$ denote the stabilizer in $L$ of the point $vP_J \in X_{vP_J}$.
Then the preimage of $X_{vP_J}$ in $G/B$ is a spherical $L$-variety if and only if  $X_{uB}$ is a spherical $K$-variety.
\end{Proposition}

\begin{proof}
Since $X_{wB}$ is a smooth Schubert variety, 
 $X_{vP_J}$ is a smooth Schubert variety in $G/P_J$, and $X_{uB}$ is a smooth Schubert variety in $G/B$. 
We can assume that $X_{vP_J}$ is a homogeneous space for $L$, hence, it is of the form $L/K$ where $K$ is the stabilizer of 
the point $vP_J$ in $L$.  
Let us denote by $X$ the preimage of $X_{vP_J}$ under the canonical morphism $p_{B,P_J}: G/B \to G/P_J$. Then  
\[
X=p_{B,P_J}^{-1}(X_{vP_J}) = X_{vw_{0,P_J}B},
\]
where $w_{0,P_J}$ is the maximal element of $W_J$. 
Then $p_{B,P_J} \vert_X : X\to X_{vP_J}\cong L/K$ is an $L$-equivariant morphism.
By Proposition~\ref{P:A1}, we see that $X$ is isomorphic to $L*_KX_{uB}$. 
By Proposition~\ref{P:A2}, we see that $X$ is smooth. 
Finally, by Proposition~\ref{P:Avdeev}, $X$ is a spherical $L$-variety if and only if $X_{uB}$ is a spherical $K$-variety. 
\end{proof}

We proceed with a simple but a general lemma. 

\begin{Lemma} \label{L:simpleobservation}
Let $f: X\to Y$ be a $G$-equivariant morphism between two $G$-varieties $X$ and $Y$.
If $Z$ is a subvariety of $Y$, then $Stab_G(Z) = Stab_G( f^{-1}(Z))$.
\end{Lemma}
\begin{proof}

Let $g$ be an element from $Stab_G(Z)$, and let $y$ be an element from $f^{-1}(Z)$. 
To check that $gy$ lies in $f^{-1}(Z)$, it suffices to check that $f(gy)\in Z$. Clearly, this is true by the equivariance of $f$. 
This proves the inclusion ``$\subseteq$''.

Conversely, let $g$ be an element of the group $Stab_G( f^{-1}(Z))$. 
We will show that $gz\in Z$ for every $z\in Z$. 
Let $a$ be an element from $f^{-1}(z)$. 
Then $g\cdot z = g\cdot f(a)  = f( g a)$. 
But $ga$ is an element of $f^{-1}(Z)$, therefore, $f(ga)\in Z$. 

\end{proof}

\begin{Remark}\label{R:HongMok}
Let $X$ be a smooth Schubert variety in a Grassmannian $G/P$.
We assume that $P$ is a maximal parabolic subgroup corresponding to a long root. 
Then by~\cite[Proposition 3.2]{HongMok}, we know that $X$ is a projective homogeneous space for a Levi subgroup $H$ of $G$. 
In particular, it is of the form $H/K$ for some parabolic subgroup $K$ of $H$. 
We will show that $K$ is a maximal parabolic subgroup of $H$.
Indeed, a projective homogeneous space is a Grassmannian if and only if its Picard number is 1. 
There is a surjective map of abelian groups \hbox{$\text{Pic}(G/P)\to \text{Pic}(X)$}, see~\cite{Mathieu}. 
Since $X$ is a smooth Schubert subvariety, $\text{Pic}(X)$ is a free abelian group, and furthermore, the former group is 
a free abelian group of rank 1.  
Therefore, the rank of $\text{Pic}(X)$ is also 1, hence, $X$ is a Grassmannian. 
\end{Remark}

\section{Smooth Schubert Divisors}\label{S:Divisors}

Let $w$ be an element from $W$.
A (right) BP decomposition of $w$, $w=vu$ ($v\in W^J,u\in W_J$), is called a {\em right chain BP decomposition} if 
$[id,v]\cap W^J$ is a chain. 
The {\em left chain BP decomposition} is defined as follows. 
Let $w=uv$ ($u\in W_J,v\in W^J$)  be a left parabolic decomposition of $w$ with respect to $J\subset S$.
Recall that $w=uv$ is a left BP decomposition if the Poincar\'e polynomial of $X_{Bw}$ is equal to the product of 
the Poincar\'e polynomials of $X_{Bu}$ and $X_{P_Jv}$. 
If, in addition, $[id,v]\cap W^J$ is a chain, then $w=uv$ is said to be a {\em left chain BP decomposition}.

\begin{Remark}\label{R:leftright}
The left (chain) BP decomposition of $w$ is a right (chain) BP decomposition of $w^{-1}$. 
Note also that $X_{wB}$ is a (rationally) smooth Schubert variety if and only if $X_{w^{-1}B}$ is a (rationally) smooth Schubert variety in $G/B$,
see~\cite[Corollary 4]{Carrell11}.
\end{Remark}

In types $\text{A},\text{B},\text{C}$, and $\text{G}_2$, every rationally smooth element $w\in W$ has a chain BP decomposition, 
but this is not true in types $\text{D},\text{E}_6,\text{E}_7,\text{E}_8$, and $\text{F}_4$. 
Nevertheless, in every type other than $\text{E}$, 
a rationally smooth element is maximal or has a chain decomposition.
For this reason we split our next result into two theorems. 
The first theorem involves the types which are not of type $\text{E}$.

\begin{Remark}
The conclusion of Theorem~\ref{T:part1} is not exclusive in the sense that there are smooth Schubert varieties $X_{wB}$ 
where $w$ is the maximal element of a parabolic subgroup and still $X_{wB}$ has a smooth Schubert divisor. 
For example, the full flag variety $G/B$ has a smooth Schubert divisor if and only if $G\in \{\mathbf{SL}_n,\mathbf{Sp}_n\}$, see~\cite[Proposition 4.12]{Carrell94}.
\end{Remark}

\begin{Theorem}\label{T:part1}
Let $X_{wB}$ be a smooth Schubert variety in $G/B$, where $G$ is not of type $\text{E}_i$, where $i\in \{6,7,8\}$.
Then $w\in W$ is a maximal element of a parabolic subgroup, hence $X_{wB}$ is a homogeneous spherical $L(w)$-variety, or 
$X_{wB}$ has a smooth Schubert divisor. 
\end{Theorem}

\begin{proof}
In the types that we listed in our hypothesis, a rationally smooth element $w\in W$ is either a maximal element of a parabolic subgroup, 
or it is has a chain decomposition, see~\cite[Theorem 4.3]{Slofstra15}. 
If $w$ is a maximal element of a parabolic subgroup, then we are in the situation of Proposition~\ref{P:U_w^-}, 
hence $X_{wB}$ is a homogeneous spherical $L(w)$-variety.
So, without loss of generality, we assume $X_{wB}$ is not a homogeneous space.

Let $X_{wB}$ be a smooth Schubert variety, 
where $w$ has a right chain BP decomposition $w= vu$ where $v\in W^J,u\in W_J$. 
Let $\pi$ denote the restriction of the map $p_{B,P_J} : G/B \to G/P_J$ onto $X_{wB}$. 
Then we know that $\pi(X_{wB})=X_{vP_J}$ is a smooth Schubert variety as well. 
Since $[id, v]\cap W^J$ is a chain, $X_{vP_J}$ is a projective space. 
As a projective space, $X_{vP_J}$ has a unique smooth Schubert divisor, $X_{v'P_J}\subset X_{vP_J}$, where $sv'=v$ with $\ell(v)=\ell(v')+1$. 
Note that $v'$ is an element of $W^J$, and that $X_{v'P_J}$ is a projective space as well. 
At the same time, since $\pi: X_{wB}\to X_{vP_J}$ is a faithfully flat and smooth morphism, 
the pull back of $X_{v'P_J}$ in $X_{wB}$ is a smooth Schubert divisor, $X_{w_1B} \subset X_{wB}$. 
This finishes the proof of our assertion.
\end{proof}

\begin{Remark}
In the proof of Theorem~\ref{T:part1}, since $\pi$ is $P(w)$-equivariant, we have $w_1= sw = s vu = v'u$. 
Notice that we can repeat this argument with a right chain BP decomposition of $w^{-1}$ to find a smooth Schubert divisor $X_{w_2^{-1}B}$
of $X_{w^{-1}B}$ with $s'w_2^{-1}=w^{-1}$ for some $s'\in D_L(w^{-1})$. 
Then $X_{w_2B}$ is a smooth Schubert divisor of $X_{wB}$, and, furthermore, $s'\in D_R(w)$.
In other words, when $X_{wB}$ is smooth and $w$ is not a maximal element of a parabolic subgroup, we have shown 
that it always contains smooth Schubert divisors $X_{w_1B}$ and $X_{w_2B}$ with 
$w_1=sw$ and $w_2 = ws'$, though it is possible that $w_1=w_2$. 
\end{Remark}

The following lemma, whose proof can be found in~\cite[Lemma 4.8]{RichmondSlofstra}, will be useful.

\begin{Lemma}\label{L:Lemma4.8}
Let $K$ be a subset of $S$, and let $v$ be an element from $W^K$. 
Let $G_{S(v)}$ be a Levi subgroup of $P_{S(v)}$, and let $P_{S(v),S(v)\cap K}$ 
be the parabolic subgroup of $G_{S(v)}$ that is generated by $S(v)\cap K$. 
Finally let $X_{vP_{S(v),S(v)\cap K}}$ denote the Schubert variety in $G_{S(v)}/P_{S(v),S(v)\cap K}$ that is indexed by $v \in W_{S(v)}^{S(v)\cap K}$.
Then the canonical inclusion map $G_{S(v)} / P_{S(v),S(v)\cap K} \hookrightarrow G/K$ induces an isomorphism 
$X_{vP_{S(v),S(v)\cap K}}\longrightarrow X_{vP_K}$.
\end{Lemma}

\begin{Proposition}\label{P:reduced}
We maintain the notation from Lemma~\ref{L:Lemma4.8}. 
Let $L$ denote the standard Levi subgroup of $Stab_{G_{S(v)}}(X_{S(v),S(v)\cap K})$, and let $L(vw_{0,K})$ denote the standard Levi subgroup of 
$Stab_G(X_{vP_K})$. 
If $X_{vP_{S(v),S(v)\cap K}}$ is a spherical $L$-variety, then $X_{vP_K}$ is a spherical $L(vw_{0,K})$-variety in $G/P_K$. 
\end{Proposition}
\begin{proof}
Since $G_{S(v)}$ is a subgroup of $G$, by Lemma~\ref{L:Lemma4.8}, 
$i: G_{S(v)} / P_{S(v),S(v)\cap K} \hookrightarrow G/K$ is a $G_{S(v)}$-equivariant embedding. 
In particular, the stabilizer subgroup $Stab_{G_{S(v)}}(X_{vP_{S(v),S(v)\cap K}})$ is contained in $Stab_G(X_{vP_K})$. 
Clearly, if $X_{vP_{S(v),S(v)\cap K}}$ is a spherical $L$-variety, then $X_{vP_K}$ is a spherical $L$-variety. It is also evident that we have the inclusion $L\subseteq L(vw_{0,K})$. 
Thus, $X_{vP_K}$ is a spherical $L(vw_{0,K})$-variety as well.
\end{proof}

Next, to handle the type $\text{E}$ cases, we will review the relevant part of~\cite[Theorem 4.3]{Slofstra15}.
We label the roots of the Dynkin diagram of $\text{E}_8$ as in Figure~\ref{F:E8}:
\begin{figure}[h]
\begin{center}
\begin{tikzpicture}[scale=.39pt]
\node at (-6,0) (a8) {$8$};
\node at (-4,0) (a7) {$7$};
\node at (-2,0) (a6) {$6$};
\node at (0,0) (a5) {$5$};
\node at (2,0) (a4) {$4$};
\node at (4,0) (a3) {$3$};
\node at (6,0) (a2) {$2$};
\node at (2,2) (a1) {$1$};

\draw[-, thick] (a8) to (a7);
\draw[-, thick] (a7) to (a6);
\draw[-, thick] (a6) to (a5);
\draw[-, thick] (a5) to (a4);
\draw[-, thick] (a4) to (a3);
\draw[-, thick] (a3) to (a2);
\draw[-, thick] (a4) to (a1);
\end{tikzpicture}
\caption{The labeling of the nodes of $\text{E}_8$.}\label{F:E8}
\end{center}
\end{figure}

Let $S_k$ denote the set $\{s_1,\dots , s_k\}$, where $s_i$ is the simple reflection corresponding to the root labeled $i$ 
in Figure~\ref{F:E8}. Then the Weyl group $W$ of $\text{E}_8$ is generated by $S_8$. 
We will view $\text{E}_6$ and $\text{E}_7$ as subgroups of $\text{E}_8$, where the Weyl group of $\text{E}_6$
is given by the parabolic subgroup $W_{S_6}$ in $W$, 
and the Weyl group of $\text{E}_7$ is given by the parabolic subgroup $W_{S_7}$ in $W$. 
Let $J_k$ denote $S_k \setminus \{s_2\}$. 
Let $\tilde{u}_k$ denote the maximal element of $W_{J_k}$, and let $\tilde{v}_k$ denote the maximal element of $W_{S_k}^{J_k}$.
Finally, we set $w_{kl}:= \tilde{v}_l \tilde{u}_k$.

\begin{Theorem}\label{T:part2}
Let $X_{wB}$ be a smooth Schubert variety in $G/B$, where $G$ is of type $\text{E}_i$, where $i\in \{6,7,8\}$.
Then $w\in W$ is a maximal element of a parabolic subgroup, hence $X_{wB}$ is $L(w)$-spherical, or 
$X_{wB}$ has a smooth Schubert divisor. 
\end{Theorem}

\begin{proof}

A smooth element $w\in W$ is either a maximal element of a parabolic subgroup, 
or it is an element of $\{w_{kl}, w_{kl}^{-1}:\ 5\leq l < k \leq 8\}$, see~\cite[Theorem 4.3]{Slofstra15}. 
If $w$ is a maximal element of a parabolic subgroup, then we are in the situation of Proposition~\ref{P:U_w^-}, 
hence $X_{wB}$ is a spherical $L(w)$-variety.
So, without loss of generality, we assume that $w\in \{w_{kl}, w_{kl}^{-1}:\ 5\leq l < k \leq 8\}$.

First we assume that $w= w_{kl}=\tilde{v}_l \tilde{u}_k$ for some $5\leq l < k \leq 8$. 
Since $\tilde{v}_l$ is the maximal element of $W_{S_l}^{J_l}$ and $\tilde{u}_k$ is the maximal element of $W_{S_k}$,
the support of $w_{kl}$ is contained in $J_k$. 
Let $G_{J_k}$ denote the reductive subgroup in $G$ that corresponds to the sub-Dynkin diagram $J_k$. 
By Lemma~\ref{L:Lemma4.8}, $X_{w_{kl}B}$ is a smooth $L(w_{kl})$-variety in $G/B$ if and only if 
$X_{w_{kl}B'}$ is a smooth Schubert variety in $G_{J_k}/B'$, where $B'=B\cap G_{J_k}$.
But $G_{J_k}$ is a semisimple algebraic group of type $\text{D}_{k-1}$, 
and by Theorem~\ref{T:part1} we know that every smooth Schubert variety in type D is either a homogeneous space, hence it is spherical,
or it has a smooth Schubert divisor.  
Thus, in this case our assertion is proved. For, $w=w_{kl}^{-1}$ the arguments are similar, so the proof of our theorem is finished.

\end{proof}

\noindent Now Theorem \ref{T:maintheorem} follows immediately from Theorem \ref{T:part1} and Theorem \ref{T:part2}.

\section{Spherical Schubert Varieties of Type $\text{A}_3$}\label{S:FinalA3}

In this section of our paper, we will analyze the sphericality of 
Schubert varieties in the full flag varieties of $\mathbf{SL}_3$ and $\mathbf{SL}_4$. 
First we have a general remark.

\begin{Remark} \label{R:simple1}
Let $P$ be a parabolic subgroup in $G$. 
For every $w$ from $W^P$, the canonical projection map $p_{B,P}\vert_{X_{wB}} : X_{wB} \to X_{wP}$ is a birational morphism. 
In particular, since $p_{B,P}\vert_{X_{wB}}$ is equivariant with respect to $L(w)$,
$X_{wB}$ is a spherical $L(w)$-variety if and only if $X_{wP}$ is a spherical $L(w)$-variety. 
\end{Remark}

Another general fact that we will use in the sequel without further mentioning is that the Schubert varieties which have dimension $\leq 1$ are homogeneous, hence, spherical varieties.

\subsection{The case of $\mathbf{SL}_3$.}

Let $G$ denote $\mathbf{SL}_3$, and let $B$ denote the Borel subgroup of upper triangular matrices in $G$. 
In this case, we have six Schubert varieties $X_{123B}, X_{213B},X_{132B},X_{231B},X_{312B},X_{312B}$.
For $w\in \{123,213,132,321\}$, the Schubert variety $X_{wB}$ is $L(w)$-spherical.
Indeed, $X_{123B}$ is a point, $X_{213B}$ and $X_{132B}$ are $B$-stable rational curves, and $X_{321B}$ is the whole 
flag variety $G/B$. 
For $i\in \{1,2\}$, let $P_i$ denote the minimal parabolic subgroup corresponding to $s_i$. 
Then, it is easy to verify that $X_{231P_1} = G/P_1$ and $X_{312P_2} = G/P_2$. 
Since $p_{B,P_1} : G/B \to G/P_1$ is a $G$-equivariant projection, and since $X_{231B}$ is birationally isomorphic to $X_{231P_1}$, 
$X_{231B}$ is $L(231)$-spherical if and only if $X_{231P_1}$ is $L(231)$-spherical.

It follows from item 1 in Remark~\ref{R:Littelmann's} that $G$ acts spherically 
on $G/P_I\times G/P_J$ for every pair of subsets $(I,J)$ from $S$ with $I \neq \emptyset$. 
Since $L(231)$ corresponds to the set of simple reflections $I=\{s_{1}\}$, $G/P_1$ is a spherical $L(231)$-variety. Hence, $X_{231P_1}$ is a spherical $L(231)$-variety. 
The proof of the fact that $X_{312B}$ is a spherical $L(312B)$-variety is similar, so, we omit it. 
We summarize these observations as follows.

\begin{Theorem}
Let $G$ denote $\mathbf{SL}_3$, and let $B$ denote the Borel subgroup of upper triangular matrices in $G$. 
Then, every Schubert variety $X_{wB}$ in $G/B$ is a spherical $L(w)$-variety. 
\end{Theorem}

\subsection{The case of $\mathbf{SL}_4$.}

Let $G$ denote $\mathbf{SL}_4$, and let $B$ denote the Borel subgroup of upper triangular matrices in $G$. We will follow our notation from Example~\ref{E:S4}.
Then the set of simple Coxeter generators is given by $S=\{s_1,s_2,s_3\}$. 
It follows from items 1 and 2 in Remark~\ref{R:Littelmann's} that $G$ acts spherically 
on $G/P_I\times G/P_J$ for every pair of subsets $(I,J)$ from $S$ 
with $|I| + |J| \geq 3$.
In particular, if $L$ is a Levi subgroup of $G$ such that $T\subsetneq L$, then 
$L$ acts spherically on $G/P_J$ for every $J\subseteq S$ with $|J|\geq 2$. 
Hence, if $X_{wP_J}$ is a Schubert variety in $G/P_J$, where $|J| =2$, then 
it is a spherical $L(w)$-variety, where $L(w)$ is the Levi subgroup of $Stab_G(X_{wP_J})$.
We are now ready to prove that every Schubert variety in any partial variety of $\mathbf{SL}_4$ is spherical. 

\begin{Theorem}\label{T:allspherical}
Let $X_{wP}$ be a Schubert variety in $G/P$, where $G=\mathbf{SL}_4$ and $P$ is a parabolic subgroup of $G$.
If $L(w)$ is a Levi subgroup of $Stab_G(X_{wP})$, then $X_{wP}$ is a spherical $L(w)$-variety. 
\end{Theorem}

\begin{proof}

The main idea of our proof which actually applies in the broader setup of spherical varities is that 
if $X_{wB}$ ($w\in W$) is a spherical $L(w)$-variety, then any Schubert subvariety $X_{vB}\subseteq X_{wB}$
which is also an $L(w)$-variety, is spherical as well. 
In this regard, we will prove that the Schubert varieties corresponding to the top elements of the posets in Figure~\ref{F:S4-2} are spherical with respect to the Levi subgroups of their stabilizing parabolic subgroups.
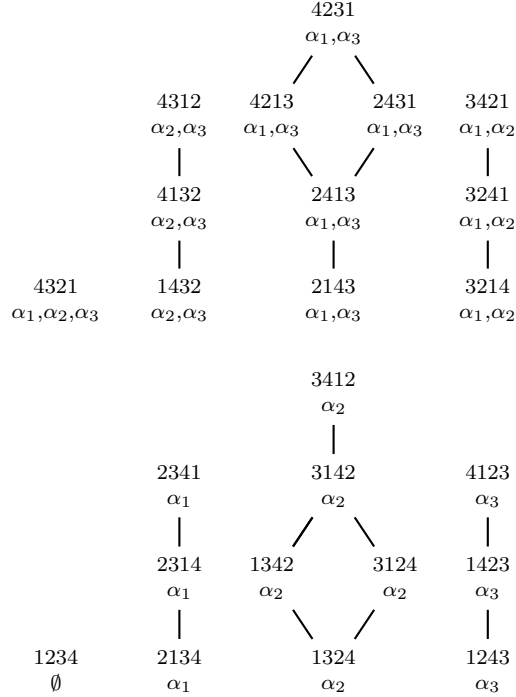
\begin{figure}[htp]
\begin{center}
\begin{tikzpicture}[scale=.41pt]

\begin{scope}[xshift=-6cm, yshift=6cm]
\node at (0,0) (g) {${4321}\atop {\alpha_1,\alpha_2,\alpha_3}$};
\end{scope}

\begin{scope}[xshift=-2cm, yshift=6cm]
\node at (0,6) (f1) {${4312}\atop{\alpha_2,\alpha_3}$};
\node at (0,3) (e1) {${4132}\atop{\alpha_2,\alpha_3}$};
\node at (0,0) (d1) {${1432}\atop{\alpha_2,\alpha_3}$};
\draw[-, thick] (d1) to (e1);
\draw[-, thick] (e1) to (f1);
\end{scope}

\begin{scope}[xshift=3cm, yshift=6cm]
\node at (0,9) (f2) {${4231}\atop{\alpha_1,\alpha_3}$};
\node at (-2,6) (e2) {${4213}\atop{\alpha_1,\alpha_3}$};
\node at (2,6) (e4) {${2431}\atop{\alpha_1,\alpha_3}$};
\node at (0,3) (d3) {${2413}\atop{\alpha_1,\alpha_3}$};
\node at (0,0) (c3) {${2143}\atop{\alpha_1,\alpha_3}$};
\draw[-, thick] (c3) to (d3);
\draw[-, thick] (d3) to (e4);
\draw[-, thick] (d3) to (e2);
\draw[-, thick] (e2) to (f2);
\draw[-, thick] (e4) to (f2);
\end{scope}

\begin{scope}[xshift=8cm, yshift=6cm]
\node at (0,6) (f3) {${3421}\atop{\alpha_1,\alpha_2}$};
\node at (0,3) (e5) {${3241}\atop{\alpha_1,\alpha_2}$};
\node at (0,0) (d5) {${3214}\atop{\alpha_1,\alpha_2}$};
\draw[-, thick] (d5) to (e5);
\draw[-, thick] (e5) to (f3);
\end{scope}

\begin{scope}[xshift=-6cm, yshift=-6cm]
\node at (0,0) (a) {${1234}\atop{\emptyset}$};
\end{scope}

\begin{scope}[xshift=-2cm, yshift=-6cm]
\node at (0,6) (d6) {${2341}\atop{\alpha_1}$};
\node at (0,3) (c5) {${2314}\atop{\alpha_1}$};
\node at (0,0) (b3) {${2134}\atop{\alpha_1}$};
\draw[-, thick] (b3) to (c5);
\draw[-, thick] (c5) to (d6);
\end{scope}

\begin{scope}[xshift=3cm, yshift=-6cm]
\node at (0,9) (e3) {${3412}\atop{\alpha_2}$};
\node at (0,6) (d4) {${3142}\atop{\alpha_2}$};
\node at (-2,3) (c2) {${1342}\atop{\alpha_2}$};
\node at (2,3) (c4) {${3124}\atop{\alpha_2}$};
\node at (0,0) (b2) {${1324}\atop{\alpha_2}$};
\draw[-, thick] (b2) to (c2);
\draw[-, thick] (b2) to (c4);
\draw[-, thick] (c2) to (d4);
\draw[-, thick] (c2) to (d4);
\draw[-, thick] (c4) to (d4);
\draw[-, thick] (d4) to (e3);
\end{scope}

\begin{scope}[xshift=8cm, yshift=-6cm]
\node at (0,0) (b1) {${1243}\atop{\alpha_3}$};
\node at (0,3) (c1) {${1423}\atop{\alpha_3}$};
\node at (0,6) (d2) {${4123}\atop{\alpha_3}$};
\draw[-, thick] (b1) to (c1);
\draw[-, thick] (c1) to (d2);
\end{scope}
\end{tikzpicture}
\caption{Schubert varieties in $\mathbf{SL}_4/\mathbf{B}_4$ with the same stabilizing Levi subgroups.}\label{F:S4-2}
\end{center}
\end{figure}

Clearly, for the Schubert varieties $X_{4321B} \cong G/B$ and $X_{1234B}\cong pt$ there is nothing to prove. 
For the Schubert varieties $X_{4312B}$, $X_{4231B}$, and $X_{3421}$, the stabilizing parabolic subgroups are given by $P_{\{s_1,s_2\}}$, $P_{\{s_1,s_3\}}$, and $P_{\{s_2,s_3\}}$, respectively. 
The Levi subgroups $L_{\{s_1,s_2\}}$ and $L_{\{s_2,s_3\}}$ act spherically on the whole flag variety $G/B$;
this follows from the first item in Remark~\ref{R:Littelmann's}. 
To show that $X_{4231B}$ is a spherical $L(4231)$-variety, we will argue by using Kempf's lemma~\ref{L:Kempf}.

Let $w$ denote the permutation $4231$, and let $w_0$ denote the longest element in $S_4$, that is, $w_0 = 4321$. 
Then $w_0 s_2 = w$. 
Let $P_2$ denote the minimal parabolic subgroup defined by $P_2 = B \cup B s_{2}B$.
By Lemma~\ref{L:Kempf}, it is easy to verify that 
the canonical projection $p_{B,P_2} : G/B \to G/P_2$ is a $\PP^1$-fibration, 
and the restriction $p_{B,P_2} \vert_{X_{wB}} : X_{wB} \to X_{w_0P_2}=G/P_2$ is a surjective birational morphism.
Evidently, $p_{B,P_2} \vert_{X_{wB}}$ is an $L(w)$-equivariant morphism. 
Therefore, $X_{wB}$ is a spherical $L(w)$-variety if and only if $G/P_2$ is a spherical $L(w)$-variety. 
Since $L(w)$ corresponds to the set of simple roots $I=\{s_{1},s_{3}\}$, 
and since $P_2$ corresponds to the set $J=\{s_{2}\}$, 
by the first paragraph of this subsection, we know that $G/P_2$ is a spherical $L(w)$-variety.
This finishes the proof of our assertion that $X_{4231B}$ is a spherical $L(4231)$-variety.

Next, we will show that $X_{4123B},X_{3412B}$, and $X_{2341B}$ are spherical varieties. 
To this end, we list the sets of Grassmann permutations in $W=S_4$. 
\begin{enumerate}
\item If $J_1:=\{s_2,s_3\}$, then $W^{J_1} = \{ 1234, 2134, 3124, 4123 \}$. 
\item If $J_2:=\{s_1,s_3\}$, then $W^{J_2} = \{ 1234, 1324, 1423, 2314, 2413, 3412 \}$. 
\item If $J_3:=\{s_1,s_2\}$, then $W^{J_3} = \{ 1234, 1243, 1342, 2341 \}$. 
\end{enumerate}
By Remark~\ref{R:simple1} and our discussion in the first paragraph of this subsection, 
we know that every Schubert variety of the form $X_{wB}$,
where $w\in W^{J_1}\cup W^{J_2} \cup W^{J_3}$ is a spherical $L(w)$-variety. 
In particular, $X_{4123B},X_{3412B}$, and $X_{2341B}$ are spherical varieties. 

This finishes the proof of our theorem.
\end{proof}

Some of the ideas that we used in the proof of Theorem~\ref{T:allspherical} for $\mathbf{SL}_4$ 
can be used for $\mathbf{SL}_n$ $(n\geq 5)$ as well as for some other connected semisimple algebraic groups.

\begin{Proposition}\label{P:6.4}
Let $G$ denote $\mathbf{SL}_n$ and let $B$ be the Borel subgroup of upper triangular matrices. 
Then, for every $n$, the Schubert variety $X_{w_0 s_iB}$, where $i\in \{1,n-1\}$, is a spherical $L(w_0s_i)$-variety.  
If $n\geq 5$, and $i\in \{2,\dots, n-2\}$, then the Schubert variety $X_{w_0 s_{i}B}$ is not a spherical $L(w_0 s_i)$-variety. 
\end{Proposition}

\begin{proof}
Let $P_i$ denote the minimal parabolic subgroup $P_i:= B\cup Bs_i B$. 
Let $w$ denote $w_0 s_i$. Clearly, $w_0 s_{i} < w_0$. 
Thus, as before, by Kempf's lemma, we know that the canonical projection $p_{B,P_i} : G/B \to G/P_i$ is a $\PP^1$-fibration, 
and the restriction $p_{B,P_i} \vert_{X_{wB}} : X_{wB} \to X_{w_0P_i}=G/P_i$ is an $L(w)$-equivariant surjective birational morphism.
Therefore, $X_{wB}$ is a spherical $L(w)$-variety if and only if $G/P_i$ is a spherical $L(w)$-variety. 
It is easily seen from Proposition~\ref{P:simpleresult} that $L(w)$ is generated by the maximal torus $T$ and the 
simple generators $s_{j}$, where $j\in \{1,\dots, n-1\}$ and $j\neq i$. 
In other words, if $P_I$ denotes the parabolic subgroup $P(w)$, then $|I^c| = 1$.  
If we assume that $n\geq 5$, then $P_J = P_i$ implies that $|J|=1$, hence that, $|J^c | \geq 3$. 
Now it follows from Remark~\ref{R:Littelmann's} that the only two cases, where $L(w)$ can act spherically on $G/P_J$ are given 
by $I^c = \{ s_1\}$ and $I^c = \{s_{n-1}\}$.

\end{proof}

\begin{Remark}
The {\em Lakshmibai-Sandhya criterion for smoothness} states that a Schubert variety $X_{wB}$ , where $w = w_1 w_2 \ldots w_n \in S_n$ , is smooth if and only if $w$ avoids the patterns $4231$ and $3412$. Via this criterion we see that the smooth Schubert divisors from Proposition~\ref{P:6.4} are precisely the ones that are spherical. In light of this it might seem natural to expect that smooth Schubert varieties are always spherical, especially since this is the case for smooth Schubert varieties in the Grassmannian. However, the example below shows that smoothness is not sufficient to gaurantee sphericity for Schubert varieties in the full flag variety.
\end{Remark}

\begin{Example}
\label{E:smoothspericalcounterexample}
Consider the smooth Schubert variety $X_{24687531B}$ in $\mathbf{GL}_8/B$,  
where $B$ is the standard Borel subgroup of upper triangular matrices in $\mathbf{GL}_8$.
The maximal Levi which acts on $X_{24687531B}$ is 
$L(24687531) \cong \mathbf{GL}_2 \times \mathbf{GL}_2 \times \mathbf{GL}_2 \times \mathbf{GL}_2$. 
The dimension of $X_{24687531B}$ is 16, while the dimension of any Borel subgroup in $L(24687531)$ is 12. 
Thus it is not possible for a Borel subgroup in $L(24687531)$ to have a dense orbit in $X_{24687531B}$, 
and hence $X_{24687531B}$ is not $L(24687531)$-spherical.
\end{Example}

\section{Spherical Schubert Varieties in $\text{B}_2$}\label{S:FinalB2}

In this section we will show that every Schubert variety in the flag variety of type $\text{B}_2$ 
is spherical with respect to the corresponding stabilizing Levi subgroup. 
To this end, we let $G$ denote $\mathbf{Spin}_5$, and let $B$ be a Borel subgroup in $G$. 
Let $s_1$ denote the simple reflection corresponding to the short root, and let $s_2$ denote the 
simple reflection corresponding to the long root. 
The minimal parabolic subgroups corresponding to $s_i$ ($i\in \{1,2\}$) will be denoted by $P_i$. 

\begin{Theorem}
Let $G$ denote $\mathbf{Spin}_5$, and let $B$ denote the Borel subgroup in $G$. 
Then, every Schubert variety $X_{wB}$ in $G/B$ is a spherical $L(w)$-variety. 
\end{Theorem}

\begin{proof}
The weak order diagram of the Weyl group $W$ of $(G,B)$ is depicted (by the solid lines) in the middle figure of Figure~\ref{F:B2}.

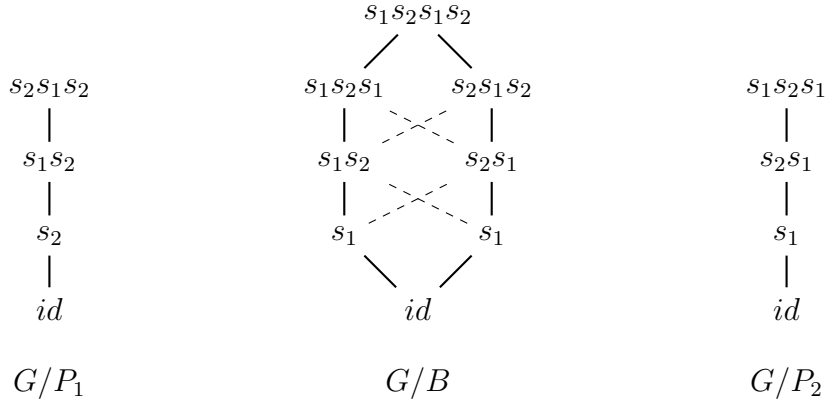
\begin{figure}[htp]
\begin{center}
\begin{tikzpicture}[scale=.49pt]

\begin{scope}[xshift=0cm]
\node at (0,-2) (a) {$G/B$};
\node at (0,0) (a1) {$id$};
\node at (-2,2) (a2) {$s_1$};
\node at (2,2) (a3) {$s_1$};
\node at (-2,4) (a4) {$s_1s_2$};
\node at (2,4) (a5) {$s_2s_1$};
\node at (-2,6) (a6) {$s_1s_2s_1$};
\node at (2,6) (a7) {$s_2s_1s_2$};
\node at (0,8) (a8) {$s_1s_2s_1s_2$};

\draw[dashed] (a2) to (a5);
\draw[dashed] (a3) to (a4);
\draw[dashed] (a4) to (a7);
\draw[dashed] (a5) to (a6);

\draw[-, thick] (a1) to (a2);
\draw[-, thick] (a1) to (a3);
\draw[-, thick] (a2) to (a4);
\draw[-, thick] (a4) to (a6);
\draw[-,thick] (a7) to (a8);
\draw[-, thick] (a6) to (a8);

\draw[-, thick] (a3) to (a5);
\draw[-, thick] (a5) to (a7);
\end{scope}

\begin{scope}[xshift=-10cm]

\node at (0,-2) (a) {$G/P_1$};

\node at (0,0) (a1) {$id$};
\node at (0,2) (a2) {$s_2$};
\node at (0,4) (a4) {$s_1s_2$};
\node at (0,6) (a6) {$s_2s_1s_2$};

\draw[-, thick] (a1) to (a2);
\draw[-, thick] (a2) to (a4);
\draw[-, thick] (a4) to (a6);

\end{scope}

\begin{scope}[xshift=10cm]

\node at (0,-2) (a) {$G/P_2$};

\node at (0,0) (a1) {$id$};
\node at (0,2) (a2) {$s_1$};
\node at (0,4) (a4) {$s_2s_1$};
\node at (0,6) (a6) {$s_1s_2s_1$};

\draw[-, thick] (a1) to (a2);
\draw[-, thick] (a2) to (a4);
\draw[-, thick] (a4) to (a6);

\end{scope}
\end{tikzpicture}
\caption{The right weak order on the Weyl group of $\text{B}_2$ is indicated by the solid lines in 
the middle figure.}\label{F:B2}
\end{center}
\end{figure}

Let $w$ denote $s_1s_2s_1$, hence, $\dim X_{wB}=3$.
By Lemma~\ref{L:Kempf}, it is easy to verify that 
the canonical projection $p_{B,P_2} : G/B \to G/P_2$ is a $\PP^1$-fibration, 
and the restriction $p_{B,P_2} \vert_{X_{wB}} : X_{wB} \to X_{w_0P_2}=G/P_2$ is a surjective birational morphism.
Evidently, $p_{B,P_2} \vert_{X_{wB}}$ is an $L(w)$-equivariant morphism. 
Therefore, $X_{wB}$ is a spherical $L(w)$-variety if and only if $G/P_2$ is a spherical $L(w)$-variety. 
Since $L(w)$ corresponds to the set of simple roots $I=\{s_1\}$, 
and since $P_2$ corresponds to the set $J=\{s_2\}$, 
by Remark~\ref{R:Littelmann's}, we know that $G/P_2$ is a spherical $L(w)$-variety. 
This argument shows that $X_{wB}$ is a spherical $L(w)$-variety. 
Notice that $L(w) = L(s_1) = L(v)$ for every $v\leq_R w$ with $v\neq id$. 
Therefore, every Schubert variety $X_{vB}$ with $v\neq id$ such that $X_{vB}\subseteq X_{wB}$ is a spherical $L(v)$-variety. 

Now let $w$ denote $s_2s_1s_2$. By arguing exactly as in the case of $s_1s_2s_1$, we see that $X_{wB}$ is a 
spherical $L(s_2)$-variety, hence the same statement holds true for every $X_{vB}$ with $1 \leq_R v\leq_R w$. 
This finishes the proof of our assertion. 
\end{proof}

\begin{Remark}
The only nonsmooth spherical variety in $G/B$, where $G=\mathbf{Spin}_5$, is $X_{s_2s_1s_2B}$. 
\end{Remark}

\section{Toric and Spherical Schubert Varieties in $\text{G}_2$}\label{S:FinalG2}

In this section we will show that half of the Schubert varieties in the flag variety of $\text{G}_2$ 
are spherical with respect to appropriate reductive subgroups. 
The following general criterion for sphericality will be useful.

\begin{Theorem}\label{T:zandw}
Let $X_{wB}$ be a Schubert variety,
and let $X_{zB}$ be a Schubert divisor in $X_{wB}$ such that 
\begin{enumerate}
\item $w= z s_\alpha$, where the simple reflection $s_\alpha$ is a generator for $L(w)$, and 
\item $X_{zB}$ is a spherical $L$-variety, where $L$ is a reductive subgroup of $L(w)$ such that 
\[
\mathbf{SL}_\alpha \ L = L\ \mathbf{SL}_\alpha,
\] 
where $\alpha$ is as in part 1.
\end{enumerate}
Then $X_{wB}$ is a spherical $L(w)$-variety. 
\end{Theorem}

\begin{proof}
Let us denote by $B_\alpha$ a Borel subgroup of $\mathbf{SL}_\alpha$, which is a subgroup of $L(w)$.
Let $Q$ denote the minimal parabolic subgroup in $G$ corresponding to the simple root $\alpha$.
By Kempf's lemma~\ref{L:Kempf},  
$X_{wB}$ is a $\PP^1$-bundle over $p_{B,Q}(X_{wB})= X_{wQ}$, and the image of $X_{zB}$ is birational onto $X_{wQ}$. 
To ease our notation, let us denote the canonical projection $p_{B,Q} \vert_{X_{wB}} : X_{wB}\to X_{wQ}$ by $p$,
and let $x_0$ denote the image of $idB$ in $p(X_{wB}) \subset G/Q$. 
Then we have $p^{-1}(x_0)= Q/B \cong \PP^1$ as a Schubert subvariety of $X_{wB}$. 
On one hand, $p^{-1}(x_0)$ is a spherical $\mathbf{SL}_\alpha$-variety, so, we may assume that $B_\alpha$ has an open orbit in $p^{-1}(x_0)$.
On the other hand, since $X_{wQ}$ is a spherical $L$-variety, there exists a Borel subgroup $B_L$ of $L$ 
such that $B_L\cdot x_0$ is open in $X_{wQ}$. 
Let us denote this $B_L$ orbit by $V$. 
By our assumption, the root subgroups $U_{\pm \alpha}$ are not contained in $L$.
Therefore, we can choose a Borel subgroup of $L(w)$ which contains both of $B_\alpha$ and $B_L$.  
Note that $p^{-1}(V)$ is $L$ stable.

Now, since $X_{wB}$ is a $\PP^1$-bundle over $X_{wQ}$, there exists an open neighborhood $V'$ of $x_0=idQ$ in $X_{wQ}$
such that $p^{-1}(V') \cong V'\times Q/B$. 
In particular, we have the isomorphism of quasiprojective varieties $p^{-1}(V'\cap V)$ and $V\cap V' \times Q/B$. 
The latter variety is a subvariety of $V\times Q/B$, on which $B_L\times B_\alpha$ acts with an open orbit. 
Since $L$ and $B_\alpha$ commute, the multiplication map 
\[
V \times Q/B \longrightarrow X_{wB}
\]
is a $B_L \times B_\alpha$ equivariant, birational morphism. 
But this implies that the Borel subgroup of $L(w)$ has an open orbit in $X_{wB}$. 
This finishes the proof of our assertion. 
\end{proof}

Before we continue with the Schubert varieties for $\text{G}_2$, we will mention another general result that is due to 
Karuppuchamy.
\begin{Lemma}[Theorem 2~\cite{Karuppuchamy}]~\label{L:Karu}
Assume that $G$ is an almost simple, simply connected complex linear algebraic group. 
If $w\in W$ is a product of distinct simple reflections, then $X_{wB}$ is a toric variety for a quotient of the maximal torus $T$ of $B$. 
\end{Lemma}
Since every normal toric variety is a spherical variety, as a consequence of Lemma~\ref{L:Karu}, 
we see that if $w$ is a product of distinct simple reflections, then $X_{wB}$ is a spherical $T'$-variety, 
where $T'$ is a quotient of the maximal torus $T$. 
It follows that $X_{wB}$ is a spherical $L(w)$-variety.

Let $G$ denote a complex simple algebraic group of type $\text{G}_2$. 
Let $S= \{s_1,s_2\}$ denote the set of simple reflections, where $s_2$ corresponds to the long root and $s_1$ corresponds to the short root. 
Associated with these two elements, we have two maximal (and minimal) parabolic subgroups
$P_1$ and $P_2$ in $G$. Then $B= P_1\cap P_2$. 
By~\cite[Theorem 2.4]{BilleyPostnikov}, we know that there are only seven smooth Schubert varieties in $G/B$; 
they are indexed by the elements of the set 
\[
\{ id, s_1,s_2, s_1s_2, s_2s_1, s_2s_1s_2, s_1s_2s_1s_2s_1s_2\}. 
\]
The right and the left weak orders on the Weyl group of $\text{G}_2$ is depicted in Figure~\ref{F:G2}.

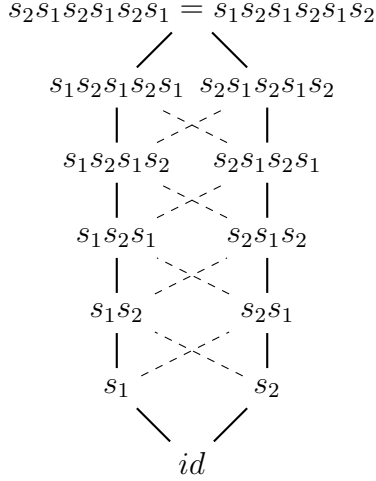
\begin{figure}[htp]
\begin{center}
\begin{tikzpicture}[scale=.5pt]

\begin{scope}[xshift=0cm]
%\node at (0,-2) (a) {$G/B$};
\node at (0,0) (a1) {$id$};
\node at (-2,2) (a2) {$s_1$};
\node at (2,2) (a3) {$s_2$};
\node at (-2,4) (a4) {$s_1s_2$};
\node at (2,4) (a5) {$s_2s_1$};
\node at (-2,6) (a6) {$s_1s_2s_1$};
\node at (2,6) (a7) {$s_2s_1s_2$};
\node at (-2,8) (a8) {$s_1s_2s_1s_2$};
\node at (2,8) (a9) {$s_2s_1s_2s_1$};
\node at (-2,10) (a10) {$s_1s_2s_1s_2s_1$};
\node at (2,10) (a11) {$s_2s_1s_2s_1s_2$};
\node at (0,12) (a12) {$s_2s_1s_2s_1s_2s_1=s_1s_2s_1s_2s_1s_2$};

\draw[dashed] (a2) to (a5);
\draw[dashed] (a3) to (a4);
\draw[dashed] (a4) to (a7);
\draw[dashed] (a5) to (a6);
\draw[dashed] (a6) to (a9);
\draw[dashed] (a7) to (a8);
\draw[dashed] (a8) to (a11);
\draw[dashed] (a9) to (a10);

\draw[-, thick] (a1) to (a2);
\draw[-, thick] (a1) to (a3);
\draw[-, thick] (a2) to (a4);
\draw[-, thick] (a4) to (a6);
\draw[-, thick] (a6) to (a8);
\draw[-, thick] (a8) to (a10);
\draw[-, thick] (a10) to (a12);

\draw[-, thick] (a3) to (a5);
\draw[-, thick] (a5) to (a7);
\draw[-, thick] (a7) to (a9);
\draw[-, thick] (a9) to (a11);
\draw[-, thick] (a11) to (a12);
\end{scope}

\begin{comment}
\begin{scope}[xshift=-10cm]
\node at (0,-2) (a) {$G/P_1$};
\node at (0,0) (a1) {$id$};
\node at (0,2) (a2) {$s_2$};
\node at (0,4) (a4) {$s_1s_2$};
\node at (0,6) (a6) {$s_2s_1s_2$};
\node at (0,8) (a8) {$s_1s_2s_1s_2$};
\node at (0,10) (a10) {$s_2s_1s_2s_1s_2$};
\draw[-, thick] (a1) to (a2);
\draw[-, thick] (a2) to (a4);
\draw[-, thick] (a4) to (a6);
\draw[-, thick] (a6) to (a8);
\draw[-, thick] (a8) to (a10);
\end{scope}
\begin{scope}[xshift=10cm]
\node at (0,-2) (a) {$G/P_2$};
\node at (0,0) (a1) {$id$};
\node at (0,2) (a2) {$s_1$};
\node at (0,4) (a4) {$s_2s_1$};
\node at (0,6) (a6) {$s_1s_2s_1$};
\node at (0,8) (a8) {$s_2s_1s_2s_1$};
\node at (0,10) (a10) {$s_1s_2s_1s_2s_1$};
\draw[-, thick] (a1) to (a2);
\draw[-, thick] (a2) to (a4);
\draw[-, thick] (a4) to (a6);
\draw[-, thick] (a6) to (a8);
\draw[-, thick] (a8) to (a10);
\end{scope}
\end{comment}

\end{tikzpicture}
\caption{The right weak order is indicated by the solid lines, and the left weak order is indicated by the dashed lines.}
\label{F:G2}
\end{center}
\end{figure}

The right weak order (resp. the left weak order) has exactly two maximal chains, 
and along each maximal chain the stabilizer subgroups are constant. 
More precisely, if $w$ and $v$ are two elements from $W$ such that 
\begin{enumerate}
\item $v\lessdot_R w$, 
\item $w\neq w_0$ and $v\neq id$,
\end{enumerate}
then, by Proposition~\ref{P:simpleresult}, $L(w) = L(v)$. 
The Levi subgroup $L(s_1)$ (resp. $L(s_2)$) is the stabilizing Levi for the Schubert varieties 
on the left (resp. the right) maximal chain.

\begin{Theorem}\label{T:G2}
Let $X_{wB}$ be a Schubert variety in $G/B$, where $G$ is of type $\text{G}_2$
Then, $X_{wB}$ is a spherical $L(w)$-variety if and only if 
$
w\in \{ id, s_1,s_2, s_1 s_2, s_2s_1, s_1s_2s_1, s_2s_1s_2, s_2s_1s_2s_2s_1s_2\}.
$
\end{Theorem}

\begin{proof}

By Lemma~\ref{L:Karu}, if $w\in \{id, s_1,s_2, s_1s_2,s_2s_1\}$, then $X_{wB}$ is toric, and hence a spherical $T$-variety.
Clearly, $X_{w_0B} \cong G/B$ is a spherical $L(w_0)$-variety also.
Now will discuss the question of sphericality for the varieties $X_{wB}$ with $w\notin \{id, s_1,s_2, s_1s_2,s_2s_1,w_0\}$.

We notice that both of the Levi subgroups $L(s_1)$ and $L(s_2)$ are isomorphic to $\mathbf{GL}_2$, hence, their Borel subgroups 
are 3 dimensional. It follows that, if $w$ is an element of the set 
\[
\{s_1s_2s_1s_2, s_2s_1s_2s_1, s_1s_2s_1s_2s_1, s_2s_1s_2s_1s_2, s_1s_2s_1s_2s_1s_2\},
\]
then $\dim X_{wB} \geq 4$, hence, $X_{wB}$ is not a spherical $L(s_i)$-variety ($i\in \{1,2\}$).
Thus, to finish our proof, we must show that $X_{wB}$ is a spherical $L(w)$-variety for the elements of 
$\{s_1s_2s_1,s_2s_1s_2\}$. 
For $w=s_1s_2s_1$, we see from Figure~\ref{F:G2} that $X_{s_1s_2s_1B}$ has $X_{s_1s_2B}$ as a divisor,
and $s_1s_2 \leq_R s_1s_2s_1$. Further, as noted above, $X_{s_1s_2B}$ is a spherical $T$-variety, and $T$ commutes with $\mathbf{SL}_{\alpha_1}$.
Therefore, by Theorem~\ref{T:zandw}, $X_{s_1s_2s_1B}$ is a spherical $L(s_1)$-variety. 
Similarly, we see that $X_{s_2s_1s_2B}$ is a spherical $L(s_2)$-variety.
This finishes the proof of our theorem.

\end{proof}

The following statement is a straightforward consequence of Theorem~\ref{T:G2}.

\begin{Corollary}
If $X_{wB}$ is a smooth Schubert variety in $G/B$, where $G=\text{G}_2$, then $X_{wB}$ is a spherical $L(w)$-variety. 
\end{Corollary}

In conclusion, we see that every smooth Schubert variety in a (partial) flag variety of a simple algebraic group of rank 2 is spherical.

\bibliography{References}

\begin{thebibliography}{10}

\bibitem{Avdeev15}
Roman Avdeev.
\newblock Strongly solvable spherical subgroups and their combinatorial
  invariants.
\newblock {\em Selecta Math. (N.S.)}, 21(3):931--993, 2015.

\bibitem{AvdeevPetukhov}
Roman~S. Avdeev and Alexey~V. Petukhov.
\newblock Spherical actions on flag varieties.
\newblock {\em Mat. Sb.}, 205(9):4--48, 2014.

\bibitem{BilleyPostnikov}
Sara Billey and Alexander Postnikov.
\newblock Smoothness of {S}chubert varieties via patterns in root subsystems.
\newblock {\em Adv. in Appl. Math.}, 34(3):447--466, 2005.

\bibitem{Billey}
Sara~C. Billey.
\newblock Pattern avoidance and rational smoothness of {S}chubert varieties.
\newblock {\em Adv. Math.}, 139(1):141--156, 1998.

\bibitem{Borel}
Armand Borel.
\newblock {\em Linear algebraic groups}, volume 126 of {\em Graduate Texts in
  Mathematics}.
\newblock Springer-Verlag, New York, second edition, 1991.

\bibitem{Brion87}
Michel Brion.
\newblock Classification des espaces homog\`enes sph\'eriques.
\newblock {\em Compositio Math.}, 63(2):189--208, 1987.

\bibitem{Can2019}
Mahir~Bilen Can, Reuven Hodges, and Venkatramani Lakshmibai.
\newblock Toroidal schubert varieties.
\newblock {\em Algebras and Representation Theory}, 2019.

\bibitem{Carrell94}
James~B. Carrell.
\newblock On the smooth points of a {S}chubert variety.
\newblock In {\em Representations of groups ({B}anff, {AB}, 1994)}, volume~16
  of {\em CMS Conf. Proc.}, pages 15--33. Amer. Math. Soc., Providence, RI,
  1995.

\bibitem{Carrell11}
James~B. Carrell.
\newblock Smooth {S}chubert varieties in {$G/B$} and {$B$}-submodules of
  {$\mathfrak g/\mathfrak b$}.
\newblock {\em Transform. Groups}, 16(3):673--680, 2011.

\bibitem{Gasharov}
Vesselin Gasharov.
\newblock Factoring the {P}oincar\'{e} polynomials for the {B}ruhat order on
  {$S_n$}.
\newblock {\em J. Combin. Theory Ser. A}, 83(1):159--164, 1998.

\bibitem{HodgesLakshmibaiCL}
Reuven {Hodges} and Venkatramani {Lakshmibai}.
\newblock {A classification of spherical Schubert varieties in the
  Grassmannian}.
\newblock {\em arXiv e-prints}, page arXiv:1809.08003, Sep 2018.

\bibitem{HodgesLakshmibai}
Reuven Hodges and Venkatramani Lakshmibai.
\newblock Levi subgroup actions on {S}chubert varieties, induced decompositions
  of their coordinate rings, and sphericity consequences.
\newblock {\em Algebr. Represent. Theory}, 21(6):1219--1249, 2018.

\bibitem{HohlwegLabbe}
Christophe Hohlweg and Jean-Philippe Labb\'{e}.
\newblock On inversion sets and the weak order in {C}oxeter groups.
\newblock {\em European J. Combin.}, 55:1--19, 2016.

\bibitem{HongMok}
Jaehyun Hong and Ngaiming Mok.
\newblock Characterization of smooth {S}chubert varieties in rational
  homogeneous manifolds of {P}icard number 1.
\newblock {\em J. Algebraic Geom.}, 22(2):333--362, 2013.

\bibitem{Karuppuchamy}
Paramasamy Karuppuchamy.
\newblock On {S}chubert varieties.
\newblock {\em Comm. Algebra}, 41(4):1365--1368, 2013.

\bibitem{Kempf1976}
George~R. Kempf.
\newblock Linear systems on homogeneous spaces.
\newblock {\em Ann. of Math. (2)}, 103(3):557--591, 1976.

\bibitem{Kramer79}
Manfred Kr\"amer.
\newblock Sph\"arische {U}ntergruppen in kompakten zusammenh\"angenden
  {L}iegruppen.
\newblock {\em Compositio Math.}, 38(2):129--153, 1979.

\bibitem{LMS74}
V.~Lakshmibai, C.~Musili, and C.~S. Seshadri.
\newblock Cohomology of line bundles on {$G/B$}.
\newblock {\em Ann. Sci. \'Ecole Norm. Sup. (4)}, 7:89--137, 1974.
\newblock Collection of articles dedicated to Henri Cartan on the occasion of
  his 70th birthday, I.

\bibitem{LakshmibaiSandhya}
V.~Lakshmibai and B.~Sandhya.
\newblock Criterion for smoothness of {S}chubert varieties in {${\rm
  Sl}(n)/B$}.
\newblock {\em Proc. Indian Acad. Sci. Math. Sci.}, 100(1):45--52, 1990.

\bibitem{Littelmann}
Peter Littelmann.
\newblock On spherical double cones.
\newblock {\em J. Algebra}, 166(1):142--157, 1994.

\bibitem{MWZ1}
Peter Magyar, Jerzy Weyman, and Andrei Zelevinsky.
\newblock Multiple flag varieties of finite type.
\newblock {\em Adv. Math.}, 141(1):97--118, 1999.

\bibitem{MWZ2}
Peter Magyar, Jerzy Weyman, and Andrei Zelevinsky.
\newblock Symplectic multiple flag varieties of finite type.
\newblock {\em J. Algebra}, 230(1):245--265, 2000.

\bibitem{Mathieu}
Olivier Mathieu.
\newblock Formules de caract\`eres pour les alg\`ebres de {K}ac-{M}oody
  g\'en\'erales.
\newblock {\em Ast\'erisque}, (159-160):267, 1988.

\bibitem{Mikityuk86}
Ihor~V. Mikityuk.
\newblock Integrability of invariant {H}amiltonian systems with homogeneous
  configuration spaces.
\newblock {\em Mat. Sb. (N.S.)}, 129(171)(4):514--534, 591, 1986.

\bibitem{OhYoo}
Suho Oh and Hwanchul Yoo.
\newblock Bruhat order, rationally smooth {S}chubert varieties, and hyperplane
  arrangements.
\newblock In {\em 22nd {I}nternational {C}onference on {F}ormal {P}ower
  {S}eries and {A}lgebraic {C}ombinatorics ({FPSAC} 2010)}, Discrete Math.
  Theor. Comput. Sci. Proc., AN, pages 965--972. Assoc. Discrete Math. Theor.
  Comput. Sci., Nancy, 2010.

\bibitem{Perrin}
Nicolas Perrin.
\newblock On the geometry of spherical varieties.
\newblock {\em Transform. Groups}, 19(1):171--223, 2014.

\bibitem{RichmondSlofstra}
Edward Richmond and William Slofstra.
\newblock Billey-{P}ostnikov decompositions and the fibre bundle structure of
  {S}chubert varieties.
\newblock {\em Math. Ann.}, 366(1-2):31--55, 2016.

\bibitem{Slofstra15}
William Slofstra.
\newblock Rationally smooth {S}chubert varieties and inversion hyperplane
  arrangements.
\newblock {\em Adv. Math.}, 285:709--736, 2015.

\bibitem{Stembridge}
John.~R. Stembridge.
\newblock Multiplicity-free products and restrictions of {W}eyl characters.
\newblock {\em Represent. Theory}, 7:404--439, 2003.

\end{thebibliography}
\bibliographystyle{plain}

\end{document}